\documentclass[11pt,a4paper,reqno]{amsart}
\usepackage[utf8]{inputenc}
\usepackage{a4wide}
\usepackage{graphicx}
\usepackage{amsmath}
\usepackage{amsfonts}
\usepackage{amsthm, enumerate}
\usepackage{comment}
\usepackage[foot]{amsaddr}

\newtheorem{theorem}{Theorem}[section]
\newtheorem*{theoremNN}{Theorem}
\newtheorem{proposition}[theorem]{Proposition}
\newtheorem{lemma}[theorem]{Lemma}
\newtheorem{corollary}[theorem]{Corollary}

\newtheorem{assumption}{Assumption}
\newtheorem{remark}[theorem]{Remark}

\usepackage[colorlinks=true,linkcolor=blue,urlcolor=blue, citecolor=blue]{hyperref}

\title[An Euler scheme for McKean SDEs with Besov drift]{An Euler scheme for McKean SDEs with Besov drift: convergence rate and implementation}

\author{Luis Mario Chaparro J\'aquez$^{1}$}
\author{Elena Issoglio$^{2}$}
\author{Jan Palczewski$^{3}$}
\address{$^1$School of Mathematics, University of Leeds, Leeds, LS2 9JT, United Kingdom}
\address{$^2$Dept. of Mathematics ``G. Peano'', University of Turin, Via Carlo Alberto 10, 10123, Torino, Italy}
\address{$^3$Faculty of Mathematics, Wroc\l{}aw University of Science and Technology, Wybrze\.{z}e Wyspia\'{n}skiego 27, 50-370, Wroc\l{}aw, Poland}
\email{$^1$academic@lmcj.xyz}
\email{$^2$elena.issoglio@unito.it}
\email{$^3$j.palczewski@pwr.edu.pl}

\usepackage[backend=biber, style=alphabetic, url=false, isbn=false, doi=false, maxbibnames=4]{biblatex}
\renewbibmacro{in:}{}

\usepackage{xcolor}
  \definecolor{leedsGreen}{RGB}{0, 80, 47}
  \definecolor{leedsRed}{RGB}{196, 18, 48}
  \definecolor{leedsBlack}{RGB}{0, 0, 0}
  \definecolor{leedsCream}{RGB}{246, 241, 228}

\usepackage{bbm} 

\usepackage{soul} 

\newcommand{\figtitle}[1]{{\small\textbf{#1}}\par\smallskip}
\newcommand{\figsubtitle}[1]{{\scriptsize#1\par\smallskip}}

\newcommand{\holder}[1]{$#1$\nobreakdash-H\"older} 
\newcommand{\cb}{c_b}
\def \lbracket{(}

\newcommand\cF{\mathcal{F}}
\renewcommand\P{\mathbb{P}}
\newcommand\E{\mathbb{E}}

\DeclareMathOperator\sgn{sgn} 
\DeclareMathOperator*{\esssup}{ess\,sup}

\allowdisplaybreaks

\begin{document}

\begin{abstract}
We study a one-dimensional McKean-Vlasov stochastic differential equation (SDE) with a drift equal to a product of a distribution depending on the state of the process and a non-linear function depending pointwise on the law density of the solution. Building on recent well-posedness results, we propose the first implementable numerical scheme for this class of SDEs. Our approach combines mollification of the distributional drift with the Euler-Maruyama scheme and a PDE-based approximation of the law via the associated Fokker-Planck equation. We prove strong convergence of the scheme and derive an explicit rate, showing how to balance the smoothing parameter with the time discretisation. Numerical experiments confirm the applicability of our scheme and demonstrate the significant influence of the McKean interaction term on the law of the solution. 
\end{abstract}

\maketitle

\section{Introduction}
\subsection{Problem setting}
In this paper we study a McKean–Vlasov stochastic differential equation (SDE) of the form
\begin{equation}
  \begin{cases}
    X_{t} = X_{0} + \displaystyle\int_{0}^{t} F(\rho(s, X_{s})) b(s, X_{s}) ds + W_t\\
    \rho(t, \cdot) \text{ is the law density of } X_t \text{ with } \rho(0, \cdot) = \rho_0,
  \end{cases}
  \label{eq:MVSDEInit}
\end{equation}
where $W$ is a one-dimensional Brownian motion, $\rho_0$ is a given density of $X_0$, $F$ is a non-linear function, and $b$ is not a function but a distribution belonging to $C([0,T]; C^{-\hat\beta})$ with $\hat\beta \in(0,\frac12)$. 

The theoretical analysis of \eqref{eq:MVSDEInit} in this singular case (in any dimension $d$) has been recently addressed in \cite{issoglioMcKeanSDEsSingular2023}, under the assumptions that $\rho_0 \in C^{\alpha}$ for some $\alpha > \hat\beta$ and that $F$ is a bounded function with suitable regularity: $F$ and $\tilde F(z) = zF(z)$ have bounded and Lipschitz continuous derivatives. A weak solution is constructed via a martingale problem, and existence and uniqueness are shown. In the one-dimensional setting considered here, it is also possible to construct the solution on a given probability space with a fixed Brownian motion and a fixed random initial condition (see Theorem~\ref{thm:strong}). This makes it possible to study strong approximations of $X$.

Analytical solutions to McKean SDEs are rare, so numerical methods are needed in most instances. A natural and often encountered way to numerically solve a McKean SDE is to use a particle system approximation, i.e., employ propagation of chaos results and simulate a cloud of particles instead of the limiting equation. The majority of contributions in this direction assume at least locally Lipschitz coefficients and the dynamics depending sufficiently smoothly on the law of the solution, see e.g. \cite{bossyStochasticParticleMethod1997, dosreisImportanceSamplingMcKeanVlasov2018, dosreisSimulationMcKeanVlasov2022, liStrongConvergenceEuler2023}. Notable exceptions are \cite{baoApproximationsMcKeanVlasov2022} and \cite{leobacherWellposednessNumericalSchemes2022}. The first paper deals with McKean SDEs with time-homogeneous coefficients that are Lipschitz continuous in the measure variable but are required only to be Hölder continuous in space. In the second paper, the McKean SDE is one dimensional but it has a discontinuous drift (with certain structural assumptions on the measure-dependence of the drift) and a diffusion coefficient that is non-negative at the discontinuity points of the drifts. An Euler-type particle approximation makes use of a transformation to remove the singularity in the drift increasing the strong convergence order from $1/9$ for the particle approximation of the original McKean SDE to $1/2$ for the one with transformed drift. We emphasise that none of these approaches applies to McKean SDEs with distributional drifts and pointwise dependence on the law density.

\subsection{Contributions}
Although an adaptation of particle approximation is possible in our distributional setting, we pursue another (rather natural) way to simulate McKean SDEs. It consists in first solving (numerically) the Fokker-Planck equation for the law of the SDE (or its density) and then inserting it into the McKean SDE which then becomes a standard SDE. This leads to an {\em implementable} numerical scheme for \eqref{eq:MVSDEInit}. We provide its theoretical {\em strong rate of convergence} and investigate its {\em empirical} performance. To the best of our knowledge, this is the first work addressing numerical approximation of McKean SDEs with distributional coefficients.

Our scheme consists of two steps. First, we regularise the distributional drift $b$ by mollification, yielding a smooth function $b^N$, where the parameter $N$ is inversely proportional to the degree of smoothing. Second, we apply the standard Euler–Maruyama scheme to the SDE \eqref{eq:MVSDEInit} in which $b$ is replaced by $b^N$. While the convergence rate of the Euler-Maruyama scheme does not depend on $N$, the associated constant explodes as $N \to \infty$, which implies that the strong error also explodes unless $N$ is appropriately balanced with the time step. We show how to couple $N$ and the time discretisation to obtain strong convergence of the combined numerical scheme and we compute the corresponding rate.

A similar two-step approach has already been used in the literature on numerical approximation of non-McKean SDEs with a distributional drift, i.e., $F \equiv 1$. 
It is natural as numerical schemes like Euler–Maruyama are not directly applicable since the distributional drift cannot be evaluated pointwise. We are aware of only three papers that have treated such approximations: \cite{deangelisNumericalSchemeStochastic2022, chaparro2023}, which focus on one-dimensional SDEs driven by Brownian motion, and \cite{goudenegeNumericalApproximationSDEs2022}, which studies $d$-dimensional SDEs driven by fractional Brownian motion. In \cite{deangelisNumericalSchemeStochastic2022} the drift is a $(\frac12+)$-H\"older continuous function of time with values in a fractional Sobolev space of negative order $H^{-\beta}_{q, \tilde q}$ with $\beta\in(0,1/4)$. In \cite{chaparro2023} the drift is $\frac12$-H\"older continuous in time with values in a negative H\"older space $C^{-\beta}$ with $\beta\in(0,1/2)$. The paper \cite{goudenegeNumericalApproximationSDEs2022} considers a time-homogeneous drift $b$ in a negative Besov space $B^{-\beta}_\infty$ for $0 < \beta \le 1/(2H) -1$ with the Hurst coefficient $H < 1/2$, i.e., SDEs driven by the standard Brownian motion are not covered. As $H \uparrow 1/2$, the admissible range for $\beta$ shrinks to zero, further distinguishing our setting from theirs. 

While our regularisation approach is conceptually similar to those used in \cite{deangelisNumericalSchemeStochastic2022, chaparro2023}, the key difference is the McKean interaction term $F(\rho)$ in the drift, which introduces additional complexity. In particular, the law density $\rho$ is not known a priori but must be computed alongside the process. It is shown in \cite{issoglioMcKeanSDEsSingular2023} that this density exists and it satisfies the Fokker-Planck equation. However, this PDE involves distributional coefficients and cannot be solved numerically. Instead, we build our scheme around the solution $X^N$ of \eqref{eq:MVSDEInit} with $b$ replaced by its mollification $b^N$. The density $\rho^N$ of $X^N$ solves the Fokker-Planck equation with smooth coefficients. This PDE can be solved using standard methods. 

Assuming first that $\rho^N$ is known exactly, we determine the relationship between the mollification level N and the Euler time step needed to ensure strong convergence, and we compute the corresponding rate. In practice, only an approximate solution $\hat\rho^N$ to the Fokker-Planck equation is available. Therefore, we quantify how the numerical error in $\hat\rho^N$ (as a function of $N$) must be controlled in order for the convergence rate of the scheme to remain unaffected. 
This decoupling is motivated by the fact that $\rho^N$ is computed only once, independently of the number of trajectories simulated for $X^N$. Hence, it is not appropriate to balance the computational cost of computing $\rho^N$ and the cost of simulating an Euler sample path. This highlights the key feature of our approach that the density $\rho^N$ of the solution is not estimated via Monte Carlo sampling but through direct numerical solution of a PDE -- often a more efficient strategy in one dimension. In higher dimensions, strong solutions of \eqref{eq:MVSDEInit} with distributional drift are not known to exist, so the main motivation to approximate the law of $X^N$ via Monte Carlo is not relevant in the present setting.

We conclude the paper with a detailed numerical study. The drift $b$ is taken to be time homogeneous. In the spatial variable $x$ it is given as the derivative of a trajectory of a fractional Brownian motion on $\mathbb{R}$ with Hurst index $H \in (1/2, 1)$, making it a singular but time-homogeneous distribution in $C^{H-1}$. We compute the empirical convergence rate of the scheme and observe that it significantly outperforms the theoretical worst-case bound, though this observation should be interpreted with caution. Since the exact solution is unavailable, we use a highly refined approximation as a reference. We also compare the empirical law of $X^N$ with the numerically computed $\hat\rho^N$ using the Kolmogorov–Smirnov goodness-of-fit test, which provides evidence for the accuracy of the approach. Finally, we investigate the influence of the McKean term $F(\rho)$ on the law of $X$ and show that the choice of $F$ can significantly affect the resulting density.

To summarise, the key contributions of this paper are: (i) a first numerical approximation scheme for McKean SDEs with distributional drift; (ii) derivation of the theoretical strong convergence rate; (iii) a practical implementation via PDE-based approximation of the density; and (iv) detailed empirical validation.

\subsection{Literature review}
We conclude the introduction with a more detailed review of related work concerning McKean SDEs.

McKean SDEs were popularised by Sznitman's Saint-Flour lectures \cite{Sznitman}. Their coefficients typically depend on the whole law of the unknown process, for example through its mean. They arise as limits of interacting particles when the number of particles tends to infinity. The dynamics of each particle is described by an SDE whose drift and/or diffusion depend on the empirical law of the cloud of particles. One can typically prove the propagation of chaos, namely that at the limit when the number of particles goes to infinity, each particle becomes independent and its dynamics is governed by its own law in place of the empirical measure. Other types of interaction have also been studied. Of particular interest to us is the case when the particles' dynamics do not depend on the empirical measure directly but on its convolution with a smooth kernel (known as moderate interaction). \cite{oelschlager_law_1985} proved that the propagation of chaos holds  for smooth enough coefficients  and leads to McKean SDEs with pointwise dependence on the law density of the process. This pointwise dependence is the one we consider in this paper: the drift at time $t$ depends on the density  of the process $\rho(t, \cdot)$  evaluated at the current position $X_t$. 

A natural and often encountered way to numerically solve a McKean SDE is to use a particle system approximation, i.e., employ propagation of chaos results and simulate a cloud of particles instead of the limiting equation. This is particuarly useful in higher dimensions because it partly avoids the curse of dimensionality, see \cite[Section 2.1]{ben_rached_double-loop_2024}. In an early contribution to the field -- \cite{bossyStochasticParticleMethod1997} -- the authors study a McKean equation with (smooth enough) coefficients depending linearly on the law. They simulate a system with $N$ particles and prove a convergence rate of order $1/2$ in time and $\sqrt N$ in the number of particles. In \cite{dosreisImportanceSamplingMcKeanVlasov2018} the assumptions on coefficients are relaxed to allow a superlinear growth in space, maintaining  Lipschitz conditions under the Wasserstein metric. The authors develop importance sampling algorithms based on particle system approximations for computing the expectations of functionals of solutions to the SDE. Under the assumptions of superlinear growth in space, Lipschitz continuity in the measure variable and 1/2-H\"older continuity in time, \cite{dosreisSimulationMcKeanVlasov2022} presents an explicit Euler scheme using interacting particle systems; this approximation is proved to be strongly convergent with the rate $1/2$ in time. This paper also shows an implicit Euler scheme but without obtaining a convergence rate. \cite{baoApproximationsMcKeanVlasov2022} deals with McKean SDEs with time-homogeneous coefficients that are Lipschitz continuous in the measure variable but are required only to be Hölder continuous in space (but the smoothness in space cannot be relaxed for the drift and diffusion coefficients simultaneously).  The authors consider an Euler scheme for the particle system and show its convergence to the solution of the original McKean SDE also providing explicit rates. The case of only locally Lipschitz condition in space for drift and diffusion coefficients is treated in  \cite{liStrongConvergenceEuler2023} where the coefficients are assumed to be globally Lipschitz in the measure variable and have linear growth uniformly in time. The authors prove convergence of a Euler-type scheme without providing a rate. 

Higher-order schemes have also been considered in the context of particle system approximations. Two works \cite{baoFirstorderConvergenceMilstein2021, kumarExplicitMilsteintypeScheme2021} appeared at the same time and studied the case of the drift with a superlinear growth in space. Under certain regularity assumptions they obtain a convergence rate of $1$ for a Milstein scheme.

A different line of extension is pursued in \cite{leobacherWellposednessNumericalSchemes2022} where the McKean SDE is one dimensional but with a discontinuous drift (with certain structural assumptions on the measure-dependence of the drift) and a diffusion coefficient that is non-negative at the discontinuity points of the drifts. The authors design an Euler-type particle approximation that makes use of a transformation to remove the singularity in the drift. Their scheme has a strong convergence rate $1/2$, while a direct time-discretisation of the particle system without the smoothing transformation is shown to be convergent with the order $1/9$.

\subsection{Organisation of the paper}

The paper is organised as follows. In Section \ref{sec:prel} we introduce notation and functional spaces in which we work. Section \ref{sec:statement} contains assumptions, a detailed explanation of the two-step numerical scheme and a statement of the main result on the rate of convergence and its interpretation. In Section \ref{sc:analytical} we collect auxiliary (existing and new) analytical results, while regularity properties of the density $\rho^N$ are derived in Section \ref{sec:rho}. The proof of the main result (Theorem \ref{thm:conv}) is presented in Section \ref{sec:main_result} for the case when the solution $\rho^N$ of the Fokker-Planck equation is exact, and extended in Section \ref{sec:rate_approx_PDE} for the case when the solution of the Fokker-Planck equation is approximated numerically (Theorem \ref{cor:conv}). Both proofs are based on bounds on the mollified drift $B^N := F(\rho^N) b^N$ obtained in Section \ref{sec:three_bounds}. Finally, Section \ref{sec:numerics} describes the numerical implementation of the scheme and explores its properties.

\section{Preliminaries}\label{sec:prel}
Given a function $f$ of $(t,x)$  we denote by $\partial_x f$ and $\partial_{xx} f$ its first  and second derivative with respect to $x$ respectively, with the same notation even  if $f$ depends only on $x$.

We denote by  $\mathcal S: = \mathcal S(\mathbb R)$ the  space of Schwartz functions and by $\mathcal S':= \mathcal S'(\mathbb R)$ the space  of Schwartz distributions on $\mathbb R$, and by $\hat \cdot $ and $\cdot^{\vee}$ the Fourier and inverse Fourier transforms on $\mathcal S$, extended to $\mathcal S'$ by duality. The notation for the sup norm is $\|f\|_{\infty} = \sup_{x \in \mathbb R} |f(x)|$ and $\|f\|_{L^\infty} = \esssup_{x \in \mathbb R} |f(x)|$.
Let  $C^{\gamma}:=C^{\gamma}(\mathbb R)$ denote the H\"older space (or Besov space when $\gamma<0$) of functions (or Schwartz distributions when $\gamma<0$) on $\mathbb R$, which is defined for any $\gamma\in\mathbb R$ as
\[
C^\gamma :=\{f \in \mathcal S' : \|f\|_{C^\gamma} := \sup_{j \geq -1 } 2^{j\gamma}\| (\phi_j \hat f)^{\vee}\|_{L^\infty} <\infty \},
\]
where $(\phi_j)_{j \in \mathbb{N}}$ is a dyadic partition of unity; for more details, see \cite[Chapter~5]{triebel2006theory} and \cite[Section 2.7]{bahouri_fourier_2011}. When  $\gamma \in \mathbb R\setminus \mathbb N$, these spaces coincide with the classical H\"older spaces. In particular, for $\gamma \in (0,1)$, denoting the   H\"older $\gamma$-seminorm by
\begin{equation}
  \label{eq:holder_seminorm}
  [f]_\gamma := \sup_{x \neq y, |x - y| < 1} \frac{|f(x) - f(y)|}{|x - y|^\gamma},
\end{equation}
then the H\"older $\gamma$-norm 
\begin{equation}
  \label{eq:holder_norm}
  \|f\|_{\gamma} 
  := \|f\|_{\infty} +  [f]_\gamma
\end{equation}
is equivalent to $\|f\|_{C^\gamma}$. Therefore, in what follows we will identify $\|\cdot\|_{C^\gamma}$ with $\|\cdot\|_\gamma$ for $\gamma \in (0,1)$.

For a Banach space $(E, \|\cdot \|_E)$, we denote by $C_T E$ the space $C([0, T]; E)$ of continuous functions on $[0, T]$ taking values in $E$ with the norm
\begin{equation*}
  \|f\|_{C_T E} := \sup_{t \in [0, T]} \|f\|_{E}.
\end{equation*}
Similarly, for $\kappa \in (0,1)$, we write $C^\kappa_T E$ for the space $C^\kappa([0, T]; E)$ of $\kappa$-H\"older continuous functions on $[0,T]$ taking values in $E$ with the following norm
\[
\|f\|_{C^\kappa_T E} := \|f\|_{C_T E} + [f]_{\kappa,E},
\]
where
\begin{equation*}
  [f]_{\kappa,E} 
 : = \sup_{t \in [0, T], t \neq s} \frac{\|f(t) - f(s)\|_{E}}{|t - s|^{\kappa}}.
\end{equation*}
Finally we denote by $L_T^\infty E$ the space of essentially bounded functions on $[0,T]$ with values in $E$ with the norm $\|\cdot \|_{L_T^\infty E}$.

Let  $(P_t)$ denote  the semigroup associated to the operator $\frac12 \partial_{xx}$ acting on  $\mathcal S$, that is for each $t>0$ we have 
\[
\begin{array}{llll}
P_t: & \mathcal S& \to &\mathcal S \\
 & \phi &\mapsto & \int_{\mathbb R} p_t (\cdot-z) \phi(z) dz
\end{array}
\]
where $p_t$ is the heat kernel given by $p_t(z) = (2 \pi t)^{-1/2} e^{-\frac{z^2}{2t}}$. The semigroup is $(P_t)$
 is extended in the usual way to $\mathcal S'$ by dual pairing, and denoted again by $(P_t)$. The heat semigroup acts well on the full scale of Besov spaces introduced above,  providing smoothing, as illustrated by the Schauder estimates below (Lemma \ref{lem:schauder}).

\section{Problem statement}\label{sec:statement}
We start by making the following assumptions, which will be in force throughout the paper. 

\begin{assumption}\label{ass:rho0}
Let $\rho_0 \in C^{1+\nu}$ for some $\nu>0$ and $\int_{-\infty}^\infty x^2 \rho_0(dx) < \infty$.
\end{assumption}

\begin{assumption}\label{ass:smallb}
  Let $\beta \in (0, 1/2)$ and $b \in C^{1/2}_TC^{-\hat\beta}$ for some $\hat \beta \in (0,\beta)$ satisfying $\beta - \hat\beta < \nu$. 
\end{assumption}

\begin{assumption}\label{ass:FF}
Let $F$ be a non-linear function and $x\mapsto\tilde F(x) := x F(x)$ be such that
\begin{itemize}
   \item $F$ is bounded and differentiable, and $\partial_x F$ is Lipschitz and bounded,
   \item $\tilde F$ is differentiable, and $\partial_x \tilde F$ is Lipschitz and bounded.
\end{itemize}
\end{assumption}

\begin{remark}\label{rem:beta}
The convergence rate will be expressed in terms of $\beta$ and the smaller this parameter is, the faster the numerical scheme will be shown to converge. The reader should think of $\beta$ as being close to $\hat \beta$. For technical reasons, we require that $\beta-\hat\beta < \nu \wedge (1-\beta)$ which is guaranteed by Assumption \ref{ass:smallb}: indeed $\beta - \hat \beta < 1/2 < 1 - \beta$, and the bound by $\nu$ is explicitly assumed. 

We remark that Assumption \ref{ass:smallb} is often written in an equivalent way: $b \in C^{1/2}_T C^{(-\beta)+}$, where $C^{1/2}_T C^{(-\beta)+} := \bigcup_{-\beta' > -\beta}  C^{1/2}_T C^{-\beta'}$.

{The square integrability condition in Assumption \ref{ass:rho0} is required to ensure the integrability of solution so that the strong error of our Euler approximation is well defined.}
\end{remark}

In order to solve the McKean-Vlasov SDE \eqref{eq:MVSDEInit}, a key tool is the associated Fokker-Planck equation 
\begin{equation}\label{eq:FPEqn}
  \begin{cases}
    \partial_t \rho = \frac12 \Delta \rho - \partial_x[\tilde{F}(\rho)b],\\
    \rho(0) = \rho_0.
  \end{cases}
\end{equation}
In this paper we understand the Fokker-Planck equation \eqref{eq:FPEqn} through its mild formulation, namely a function $\rho$ is a solution to \eqref{eq:FPEqn} if $\rho \in C_T C^\alpha$ for some $\alpha>\beta$ and it solves the integral equation
\begin{equation}\label{eq:mild_soln_fp}
  \rho(t) = P_t \rho_0 - \int_0^t P_{t - s}\left[\partial_x \left[\tilde F(\rho(s)) b(s)\right]\right] ds, \qquad t \ge 0.
\end{equation}
We note that one can also study a weak formulation of the PDE \eqref{eq:FPEqn}. However, it was shown in \cite[Proposition 3.2]{issoglioMcKeanSDEsSingular2023} that mild and weak formulations are equivalent. 

\begin{proposition}[{\cite[Proposition 3.6]{issoglioMcKeanSDEsSingular2023}}]
  \label{prop:mildSoln}
  Under Assumptions  \ref{ass:rho0}, \ref{ass:smallb}  and  \ref{ass:FF}, PDE \eqref{eq:FPEqn}
  has a unique mild solution in the space $C_T C^\alpha$ for any $\alpha \in (\beta, 1-\beta)$.
\end{proposition}
The above result is proved in \cite{issoglioMcKeanSDEsSingular2023} under weaker assumptions, namely $\rho_0 \in C^{\alpha}$ and 
$b \in C_T C^{-\hat\beta}$ as there is no need for H\"older regularity of the drift $b$ in time.

We are now in a position to formally define what we mean by a solution to \eqref{eq:MVSDEInit} following \cite[Definition 6.2]{issoglioMcKeanSDEsSingular2023}. A triple $(X, \rho, (\Omega, \cF, (\cF_t),\mathbb P))$, in short written as $(X, \rho, \P)$, is a \emph{solution to McKean-Vlasov SDE} \eqref{eq:MVSDEInit} if 
\begin{itemize}
\item $X$ is a stochastic process on a filtered probability space $(\Omega, \cF, (\cF_t),\mathbb P)$,  
\item $(X, \mathbb P )$ is a solution to the {\em martingale problem} with distributional drift $B:=  F(\rho) b $ and $X_0 \sim \rho_0(x) dx$ in the sense of  \cite[Definition~3.1]{issoglioSDEsSingularCoefficients2023},
\item $\rho \in C_T C^{\beta}$, and $\rho(t)$ is the density of $X_t$ for all $t \in [0, T]$. 
\end{itemize}
The problem admits \emph{uniqueness} if for any two solutions $(X, \rho, \P)$ and $(X', \rho', \P')$ the densities $\rho, \rho'$ coincide as elements of $C_TC^{\beta}$ and the laws of $(X, \P)$ and $(X', \P')$ are identical.

\begin{theorem}[{\cite[Theorem 6.3 and Theorem 4.2]{issoglioMcKeanSDEsSingular2023}}]\label{thm:1}
Let Assumptions  \ref{ass:rho0}, \ref{ass:smallb}  and  \ref{ass:FF} hold. There is a unique solution to McKean-Vlasov SDE \eqref{eq:MVSDEInit} and $\rho$ is a mild solution to PDE \eqref{eq:FPEqn}.
\end{theorem}

To study the strong error of a numerical approximation to $X$, we have to be able to construct a solution $X$ on a given probability space (on the space employed to run the numerical scheme and support the random initial condition). It is possible in dimension $1$ as the following theorem shows. The limitation to dimension one comes from \cite[Ch.~IV, Thm.~3.2]{ikeda} as our construction involves an SDE with H\"older continuous diffusion coefficient. We emphasise that the strong uniqueness in a classical sense cannot be considered here as the solution to the SDE \eqref{eq:MVSDEInit} is only defined in the weak sense.

\begin{theorem}\label{thm:strong}
Let $(\Omega, \cF, (\cF_t),\mathbb P)$ be a filtered probability space supporting a Brownian motion $(W_t)$ and an independent $\cF_0$-measurable random variable $\xi$ with the density $\rho_0$. Under the assumptions of Theorem~\ref{thm:1}, there is a stochastic process $X$ adapted to the filtration $(\cF_t)$ such that $(X, \rho, \P)$ is a solution to the McKean-Vlasov SDE \eqref{eq:MVSDEInit} and $X_0 = \xi$ $\P$-a.s. Furthermore, $\E X^2_t < \infty$ for any $t \in [0, T]$.
\end{theorem}

{
\begin{proof}
We will construct $X$ as a so-called virtual solution. For $\lambda > 0$, let us consider a PDE
\begin{equation}
		\label{eq:kolmogorov}
	  	\begin{cases}
			u_t + \frac{1}{2} u_{xx} + F(\rho) b  u_x = \lambda u - F(\rho) b
	  	  	\\
	  	  	u(T) = 0
	  	\end{cases}
	\end{equation}
with the solution understood in a mild sense. As $b \in C_TC^{-\hat\beta}$, $\rho \in C_T C^\alpha$, $\alpha > \beta$ and $F \in C^1$, we have $F(\rho) b \in C_T C^{-\hat \beta}$, where we use the property of the point-wise product between a function and a distribution defined using Bony's paraproduct \cite{bony}, see also \cite[Eq.~\lbracket2.9)]{issoglioMcKeanSDEsSingular2023} for the specific application of Bony's paraproduct to the current setting. It is shown in \cite[Theorem 4.7]{issoglioPdeDriftNegative2022} that a solution $u$ belongs to $C_TC^{(2 - \hat\beta)-} $ and is unique in $C_T C^{(1+\hat\beta)+}$. Hence, $u_x$ is  \holder{\alpha} continuous for any $\alpha <1-\hat\beta$.	By \cite[Proposition 4.13]{issoglioPdeDriftNegative2022}, we have $|u_x|_\infty<1/2$ for $\lambda$ large enough and the mapping 
\begin{equation}\label{eq:phi}
    \phi(t, x):= x + u(t,x)
\end{equation} 
is invertible in the space dimension. We denote this space-inverse by $\psi(t, \cdot)$. Due to the bound on $u_x$, $\psi$ is Lipschitz in the spatial variable with constant $1/(1-1/2) = 2$ uniformly in $t$. Consider now an SDE
\begin{equation}
	\label{eq:Y_def}
	Y_t = Y_0 + \lambda \int_0^t u(s, \psi(s, Y_s)) ds + \int_0^t (  u_x (s, \psi(s, Y_t)) + 1 ) dW_s
\end{equation}
with $Y_0 = \phi(0, \xi)$. The drift is Lipschitz in $y$ uniformly in $t$ and the volatility is \holder{\alpha} in $y$ uniformly in $t$. This allows us to use arguments in the proof of \cite[Ch.~IV, Thm.~3.2]{ikeda} (see also the corollary following it) to claim the existence and uniqueness of a strong solution. By \cite[Thm.~3.9]{issoglioSDEsSingularCoefficients2023}, the process $X_t := \psi (t, Y_t)$, which we call a \emph{virtual solution}, is a solution to the martingale problem  \eqref{eq:MVSDEInit}  with distributional drift $B = F(\rho) b$ and  initial condition $X_0 = \psi(0, Y_0) = \psi(0, \phi(0, \xi)) = \xi$.\footnote{The term virtual solution comes from \cite{flandoliMultidimensionalStochasticDifferential2017}, where the authors use an analogous equation for $Y$ to define the solution of an SDE with distributional drift $b$ in a fractional Sobolev space. The equivalence of this solution with the solution defined via a martingale problem in \cite{issoglioSDEsSingularCoefficients2023} follows from their Theorem 3.9. We also note that the virtual solution does not depend on $\lambda$ thanks to the links between virtual solutions and the martingale problem; the reader is refered to \cite{issoglioSDEsSingularCoefficients2023} for a complete presentation of those links.} 

We turn attention to proving the square integrability of $Y_t$. By the Lipschitz property of $x \mapsto \phi(0, x)$, $\E Y_0^2 < \infty$ as $Y_0 = \phi(0, \xi)$ and $\E\xi^2 < \infty$. We infer $\E Y_t^2 < 0$ by an easy adaptation of the proof of \cite[Ch.~IV, Thm.~2.4]{ikeda} to time-inhomogeneous case as long as the following condition holds:
\[
(  u_x (s, \psi(s, y)) + 1 )^2 + u(s, \psi(s, Y_s))^2 \le K (1 + x^2)
\]
for some constant $K > 0$ independent of $s \in [0, T]$. This follows immediately from the boundedness of $u$ and $u_x$ as $u \in C_T C^{2-\hat\beta}$. To transfer this square integrability to $X$, we show that the mapping $\psi$ has linear growth in $y$ uniformly in $t \in [0, T]$. To this end, we write
\[
| \psi(t, y)| \le |\psi(t, y) - \psi(t, \phi(t, 0))| + |\psi(t, \phi(t, 0))| \le 2 |y - \phi(t, 0)| \le 2|y| + \sup_{t \in [0, T]} 2|\phi(t, 0)|,
\]
where $\sup_{t \in [0, T]} 2|\phi(t, 0)| = \sup_{t \in [0, T]} 2|u(t, 0)| < \infty$ by the boundedness of $u$.
\end{proof}
}

Our numerical approach is to approximate the McKean-Vlasov SDE \eqref{eq:MVSDEInit} by replacing $b$ with 
\begin{equation}\label{eq:bn}
  b^N := P_{\frac{1}{N}} b
\end{equation}
and then applying Euler scheme to solve the resulting SDE. Since the density of the solution appears in the equation, we start with the study of the solution to the following regularised Fokker-Planck equation
\begin{equation}
  \label{eq:FPEqnReg}
  \begin{cases}
    \partial_t \rho^N = \frac12 \Delta \rho^N - \partial_x[\tilde{F}(\rho^N)b^N],\\
    \rho^N(0) = \rho_0.
  \end{cases}
\end{equation}
As before, Proposition \ref{prop:mildSoln} guarantees the existence of a unique mild solution:
\begin{equation}\label{eq:mild_soln_rfp}
  \rho^N(t) = P_t \rho_0 - \int_0^t P_{t - s}\left[\partial_x \left[\tilde F(\rho^N(s)) b^N(s)\right]\right] ds, \qquad t \in [0, T].
\end{equation}
Furthermore \cite[Proposition~4.4]{issoglioMcKeanSDEsSingular2023} ensures that 
\begin{equation}\label{eq:rhoNminrho}
\lim_{N\to\infty} \|\rho^N -\rho\|_{C_TC^{\alpha}}=0 \quad \text{for any } \alpha \in (\beta, 1-\beta). 
\end{equation}

Let us denote the drift coefficient of the McKean-Vlasov SDE \eqref{eq:MVSDEInit} by $B := F(\rho) b$ and its smooth approximation by $B^N := F(\rho^N) b^N$. Let  $(\Omega, \cF, (\cF_t),\mathbb P)$ be a filtered probability space supporting a Brownian motion $(W_t)$ and let $X$ be the `strong solution' from Theorem \ref{thm:strong} so that $(X, \rho, \P)$ is  the solution to \eqref{eq:MVSDEInit}. On the same probability space, we consider the regularised McKean-Vlasov SDE:
\begin{equation}  \label{eq:MVSDEReg}
  X^N_{t} = X_{0} + \displaystyle\int_{0}^{t} B^N(s, X^N_{s}) ds + W_t.
\end{equation}
We emphasise that $X_0$ is the initial value of the solution $X$. Equation \eqref{eq:MVSDEReg} has a unique strong  solution $X^N$ which has the  density $\rho^N$.

We discretise the time interval $[0,T]$ into $m$ equal subintervals and we set $t_i:= i \frac Tm$ for $i=0, \ldots, m$. The Euler-Maruyama approximation at time $ t \in (t_i, t_{i+1}]$ of the SDE \eqref{eq:MVSDEReg}
will be
\begin{equation}
  X^{N,m}_t = X^{N,m}_{t_{i}} + B^N(t_i, X^{N,m}_{t_i}) (t - t_i )  +W_{t} - W_{t_i}.
  \label{eq:RegEMScheme}
\end{equation}
Our aim is to quantify the approximation error between $X^{N,m}$ and the solution $X$ of the McKean SDE \eqref{eq:MVSDEInit}. There are two sources of errors. The first one is the approximation of the drift $F(\rho)b$ with $F(\rho^N)b^N$. The second one is the error introduced by the Euler-Maruyama approximation to $X^N$. The larger the $N$, the better the approximation of $X$ by $X^N$, but also the bigger the error of the Euler-Maruyama solution due to a more volatile drift (i.e., with a higher amplitude of the function itself and of its derivative). Our strategy is therefore to identify an optimal dependence of $N$ on $m$, so as to maximise the strong rate of convergence of $X^{N, m}$ to $X$.

There are two main theoretical results of this paper. In the first one, it is assumed that $\rho^N$ is known exactly as presented above. In the second one, we provide conditions that a numerical solver to $\rho^N$ must satisfy in order to keep the convergence rate unchanged when $\rho^N$ is replaced by its approximation in \eqref{eq:MVSDEReg} and \eqref{eq:RegEMScheme}. Details and proofs are in Theorem \ref{thm:conv} and Theorem \ref{cor:conv}, respectively. Both theorems have the following general statement.

\begin{theoremNN}
Denote by $\hat X^{N,m}$ the Euler-Maruyama solution \eqref{eq:RegEMScheme} with $\rho^N$ or its approximation $\hat \rho^N$. For any $\lambda \in (0, 1/2 - \beta)$, setting $N(m)=m^\kappa$, where
\[
\kappa= \frac 1{(1+\beta)  + 2 (\frac12 - \beta - \lambda)^2},
\]
yields the convergence rate
\[
\sup_{t \in [0, T]} \mathbb E |\hat X^{N(m),m}_{t} - X_{t}| \leq c  m^{-\frac12 + \frac {(1+\beta)/2}{(1+\beta)  + 2 (\frac12 - \beta - \lambda)^2} }
\]
for some constant $c > 0$ and any $m$ sufficiently large so that $\|B^{N(m)} - B\|_{C_TC^{-(1/2 - \lambda)}} < 1$.
\end{theoremNN}

\begin{remark}
This convergence rate is identical, up to notational changes, to that obtained for the SDE without the McKean factor $F(\rho)$ in \cite{chaparro2023}. In particular, the following observations hold true. The function $\lambda \mapsto \frac12 - \frac {(1+\beta)/2}{(1+\beta)  + 2 (\frac12 - \beta - \lambda)^2}$ is decreasing in $\lambda$ for $\lambda \in (0, 1/2-\beta)$, so it is bounded from above by
 \[
r(\beta) = \frac 12 - \lim_{\lambda \downarrow 0} \frac {(1+\beta)/2}{(1+\beta)  + 2 (\frac12 - \beta - \lambda)^2}
= \frac12 - \frac{(1+\beta)/2}{1 +\beta + 2(\frac12-\beta)^2}.
 \]
This means that any rate strictly smaller than $r(\beta)$ is achievable by our numerical scheme. The limiting behaviour of $r(\beta)$ in extremes, i.e., when $\beta \downarrow 0$ and $\beta \uparrow 1/2$, is as follows
\[
\lim_{\beta \downarrow 0} r(\beta) = \frac12 - \frac13 = \frac16, \qquad
\lim_{\beta \uparrow 1/2} r(\beta) = \frac12 - \frac12 = 0.
\]
The reader is encouraged to consult a deeper discussion on our rate $r(\beta)$ in relation to the literature in \cite[Remark 10]{chaparro2023}. We only remark that although our rate when $\beta \downarrow 0$ appears suboptimal in comparison for the convergence rate of $1/2$ for measurable drifts, it is still compares favourably to existing results, e.g., in \cite{deangelisNumericalSchemeStochastic2022}. Our numerical results presented in Section \ref{sec:numerics} suggest that the rate is indeed suboptimal, but a stronger result would require a different approach than ours.
\end{remark}

We briefly illustrate our strategy for the proof of the main result on convergence rate. The idea is to split the error into three terms
\begin{equation}\label{eq:TriangleIneq}
  \sup_{t\in[0,T]}\mathbb E | \hat X^{N,m}_{t} - X_{t} |
  \leq \sup_{t\in[0,T]}\mathbb E | \hat X^{N,m}_{t} - \hat X^N_{t} |
  + \sup_{t\in[0,T]} \mathbb E | \hat X^N_{t} - X^N_{t} |
  +\sup_{t\in[0,T]}\mathbb E | X^N_{t} - X_{t} |,
\end{equation}
where the first one is the error of the Euler-Maruyama scheme applied to a smoothed SDE (the {\em Euler error}), the second one is the error due to the approximation of the solution of PDE \eqref{eq:FPEqnReg} with $\hat \rho^N$ (this term disappears in Section \ref{sec:main_result} where exact $\rho^N$ is used), while the third term is controlled with {\em stability estimates} on the SDE. Clearly, since the drift $B^N$ of the  SDE \eqref{eq:MVSDEReg} is smooth, a standard result on the convergence of Euler scheme could be applied to control the Euler error. The issue, however, would be that  the constant in front of the rate would depend on the properties of the drift such as its norm or the norm of its derivative in an exponential way, exploding as $N\to \infty$. To control this constant, we have to obtain new estimates for the Euler scheme error that depend  on the norms and seminorms $[B^N]_{\frac{1}{2}, L^{\infty}}$, $ \|B^N\|_{C_T L^{\infty}}$ and $\| B^N_x\|_{C_T L^{\infty}}$ in a polynomial way, rather than an exponential way.

\section{Analytical results}\label{sc:analytical}

\begin{lemma}[Schauder estimates]\label{lem:schauder}
\leavevmode
\begin{enumerate}
\item[(i)] For any $\xi \in \mathbb R$, $\zeta \geq 0$ and $t' > 0$, there is a constant $c_S^1 = c_S^1(\xi, \zeta, t')$ such that for any $f \in C^{\xi - 2\zeta}$
\begin{equation*}
\| P_t f \|_{C^\xi} \leq c_S^1 t^{-\zeta} \| f \|_{C^{\xi-2\zeta}}, \qquad t \in (0, t'].
\end{equation*}
\item[(ii)] For $\xi \in \mathbb R$,  $\zeta \in (0, 1)$ and $t' > 0$, there exists a constant $c_S^2=c_S^2(\xi, \zeta, t')$ such that $f \in C^{\xi + 2\zeta}$, 
\begin{equation*}
\| P_t f  - P_s f\|_{C^\xi} \leq c_S^2 (t - s)^{\zeta} \| f \|_{C^{\xi + 2\zeta}}, \qquad 0 \le s < t \le t'.
\end{equation*}
In particular, when $s = 0$, we obtain
\begin{equation*}
\| P_t f  - f\|_{C^\xi} \leq c_S^2 t^{\zeta} \| f \|_{C^{\xi + 2\zeta}}, \qquad 0 < t \le t'.
\end{equation*}
\end{enumerate}
\end{lemma}
For a proof of the above lemma see \cite[Theorem 2.10]{issoglio2024degenerate} with $d=0$ and $N=1$ therein.
\begin{remark}
In the paper, we will take $t = 1/N$ for $N \ge 1$, so it is enough to consider the above estimates with $t' = 1$. We will therefore omit the dependence on $t'$ when we refer to constants $c^1_S$ and $c^2_S$. 
\end{remark}

Below we recall Berstein inequality that bounds the norm of the derivative of a function (distribution) with the norm of the function itself, in a suitable Besov space. It is a simple consequence of \cite[Lemma 2.1]{bahouri_fourier_2011}.
\begin{lemma}[Bernstein inequality]\label{lem:bernstein}
  \leavevmode
  For any $\xi \in \mathbb R$ there exists $c_B = c_B(\xi)$ such that 
  \begin{equation}
    \| \partial_x f \|_{C^\xi} \leq c_B \| f \|_{C^{\xi + 1}}
  \end{equation}
  for all $f \in C^{\xi + 1}$.
\end{lemma}

We can view the function $\tilde F$ from Assumption \ref{ass:FF} as an operator on $C^\alpha$ for $\alpha \in (0, 1)$;
 for the sake of clarity, we will use $\tilde F$ to denote this operator, so that $\tilde F(f) := \tilde F(f(\cdot))$ for $f \in C^\alpha$.
Useful properties of such operators are stated in the following lemma.
\begin{lemma}[{\cite[Prop. 3.1]{issoglioNonlinearParabolicPDE2019}}]\label{lem:prop_op_F}
  Let $\varphi: \mathbb R \to \mathbb R$
  be a function such that $\varphi_x$ is Lipschitz and bounded.
  For $\alpha \in (0,1)$
  define the operator $\varphi: C^\alpha \to C^\alpha$
  as $\varphi(f)(\cdot) := \varphi(f(\cdot))$ for any $f \in C^\alpha$. 
  Then there is a constant $c_{\varphi}>0$ such that for all $f, g \in C^\alpha$
  \begin{equation*}
    \|\varphi(f) - \varphi(g)\|_{C^\alpha} \leq c_\varphi(1 + \|f\|_{C^\alpha}^2 + \|g\|_{C^\alpha}^2)^{1/2} \|f -g \|_{C^\alpha}
  \end{equation*}
  and
  \begin{equation*}
    \|\varphi(f)\|_{C^\alpha}\leq c_\varphi(1 + \|f\|_{C^\alpha}).
  \end{equation*}
\end{lemma}

 We now turn attention to a series of technical estimates.

\begin{lemma}[]\label{lem:holder_reg_pde}
  \leavevmode
  For $f \in C_T C^\xi$ for some $\xi \in \mathbb R$ let us define
  \begin{equation}\label{eqn:14}
    g(t) := \int_0^t P_{t-s} f(s) ds.
  \end{equation}
  Then $g \in C^{1/2}_{T} C^{\zeta}$  for any $\zeta < \xi+ 1$.
  Furthermore,
  $\|g\|_{C^{1/2}_{T} C^{\zeta}} \leq c_2 \| f \|_{C_T C^{\xi}}$,  where
  \begin{equation*}
    c_2 = c_2(T, \xi, \zeta)
    =
    \left(\frac{T^{1 - \frac{\zeta - \xi}{2}}}{1 - \frac{\zeta - \xi}{2}}
    + 2  + c_{S}^2 \frac{T^{1 - \frac{1 + \zeta - \xi}{2}}}{1 - \frac{1 + \zeta - \xi}{2}} \right) c_S^1.
  \end{equation*}
\end{lemma}
\begin{proof}
For clarity, we recall the definition of the norm to bound:
  \begin{equation}\label{eq:g}
    \|g\|_{C^{1/2}_{T} C^{\zeta}}
    =
    \| g \|_{C_T C^{\zeta}} + [g]_{1/2, C^\zeta}
    =
    \sup_{t \in [0,T]}\| g(t) \|_{C^{\zeta}} + \sup_{t \neq r \in [0,T]} \frac{\| g(t) - g(r) \|_{C^\zeta}}{|t - r|^{1/2}}.
  \end{equation}
We will estimate each term separately. For the first term in \eqref{eq:g}, using Jensen's inequality we have
  \begin{equation*}
    \| g(t) \|_{C^{\zeta}} = \left\| \int_0^t P_{t - s} f(s) ds \right\|_{C^{\zeta}}
    \leq \int_0^t\left\| P_{t - s} f(s) \right\|_{C^{\zeta}} ds.
  \end{equation*}
Applying Lemma \ref{lem:schauder}(i) we obtain
  \begin{equation*}
    \| g(t) \|_{C^{\zeta}} \leq c^1_S \int_0^t (t - s)^{-\left(\frac{\zeta - \xi}{2}\right)} \| f (s) \|_{C^\xi} ds
    \le  c^1_S \frac{t^{1 - \frac{\zeta - \xi}{2}}}{1 - \frac{\zeta - \xi}{2}} \| f \|_{C_T C^{\xi}},
  \end{equation*}
where $c^1_S$ is the constant from the Schauder estimate and in the last inequality we use $\zeta < \xi + 1$ and $\| f (s) \|_{C^\xi} \le \| f \|_{C_T C^{\xi}}$. We take the supremum over $t$ and conclude
  \begin{equation}
    \label{eq:gctcnu}
    \| g \|_{C_T C^{\zeta}} \leq c^1_S \frac{T^{1 - \frac{\zeta - \xi}{2}}}{1 - \frac{\zeta - \xi}{2}} \| f \|_{C_T C^{\xi}}.
  \end{equation}

For the second term in \eqref{eq:g}, let us fix  $0 \leq r < t \leq T$. We have 
  \begin{equation} \label{eq:bound_semigroups}
    \begin{aligned}
      \| g(t) - g(r) \|_{C^\zeta} &= \left\| \int_0^t P_{t - s} f(s) ds - \int_0^r P_{r - s} f(s) ds \right\|_{C^\zeta}\\
				&= \left\| \int_r^t P_{t - s} f(s) ds + \int_0^r \left[P_{t - s} f(s) - P_{r - s} f(s)\right] ds \right\|_{C^\zeta}\\
				&\leq \int_r^t \left\| P_{t - s} f(s) \right\|_{C^\zeta} ds 
				+ \int_0^r \left\| P_{r - s} \left[P_{t - r} f(s) - f(s)\right] \right\|_{C^\zeta} ds,\\
    \end{aligned}
  \end{equation}
  where the last inequality comes from the semigroup property and Jensen's inequality. For the first term on the right-hand side we apply Lemma \ref{lem:schauder}(i) to get
  \begin{equation*}
    \int_r^t \left\| P_{t - s} f(s) \right\|_{C^\zeta} ds
    \leq
    \int_r^t c_{S}^1 (t - s)^{-1/2} \left\| f(s) \right\|_{C^{\zeta - 1}} ds.
  \end{equation*}
Noting that  $\xi > \zeta - 1$, taking the supremum over $s$  and computing the integral we get
\begin{equation}\label{eq:bound_sg_1st}
\int_r^t \left\| P_{t - s} f(s) \right\|_{C^\zeta} ds
\leq \int_r^t c_{S}^1 (t - s)^{-1/2} \left\| f(s) \right\|_{C^{\xi}} ds
\leq c_{S}^1\, 2(t - r)^{1/2} \left\| f \right\|_{C_T C^{\xi}}.
\end{equation}
For the second term in  \eqref{eq:bound_semigroups} we apply Lemma \ref{lem:schauder}(i) and (ii) for some $\eta\in(0,1)$ to be specified later
\begin{equation*}
\begin{aligned}
      \int_0^r \left\| P_{r - s} \left[P_{t - r} f(s) - f(s)\right] \right\|_{C^\zeta} ds
      &\leq
      \int_0^r c_{S}^1 (r - s)^{-\eta} \left\| P_{t - r} f(s) - f(s) \right\|_{C^{\zeta - 2 \eta}} ds\\
      &\leq
      \int_0^r c_{S}^1 (r - s)^{-\eta} c_{S}^2(t - r)^{1/2} \left\| f(s) \right\|_{C^{\zeta - 2 \eta + 1}} ds\\
      &=
      c_{S}^1 c_{S}^2 (t - r)^{1/2} \left\| f \right\|_{C_T C^{\zeta - 2 \eta + 1}} \int_0^r (r - s)^{-\eta} ds\\
&=  c_{S}^1 c_{S}^2 \frac{r^{1 - \eta}}{1 - \eta} (t - r)^{1/2} \left\| f \right\|_{C_T C^{\zeta - 2 \eta + 1}},
\end{aligned}
\end{equation*}
where the last inequality comes from taking the supremum over $s$. Setting $\eta = \frac{1 + \zeta - \xi}{2} $ so as $\zeta - 2 \eta + 1 = \xi$ we obtain
\begin{equation}
    \label{eq:bound_more_semigroups}
    \int_0^r \left\| P_{r - s} \left[P_{t - r} f(s) - f(s)\right] \right\|_{C^\zeta} ds
    \leq
     c_{S}^1 c_{S}^2 \frac{r^{1 - \frac{1 + \zeta - \xi}{2}}}{1 - \frac{1 + \zeta - \xi}{2}} (t - r)^{1/2} \left\| f \right\|_{C_T C^{\xi}}.
\end{equation}
Stitching \eqref{eq:bound_sg_1st} and \eqref{eq:bound_more_semigroups} together we get
\begin{equation*}
    \|g(t) - g(r)\|_{C^\zeta} \leq 
    2 c_{S}^1 (t - r)^{1/2} \left\| f \right\|_{C_T C^{\xi}}
    +
    c_S^1 c_{S}^2 \frac{r^{1 - \frac{1 + \zeta - \xi}{2}}}{1 - \frac{1 + \zeta - \xi}{2}} (t - r)^{1/2} \left\| f \right\|_{C_T C^{\xi}}.
  \end{equation*}
Thus
\begin{equation}\label{eq:ghalfsemi}
[g]_{1/2, C^\zeta} \leq
    \left(2 c_S^1
      +
    c_S^1c_{S}^2 \frac{T^{1 - \frac{1 + \zeta - \xi}{2}}}{1 - \frac{1 + \zeta - \xi}{2}}\right)
    \left\| f \right\|_{C_T C^{\xi}}.
  \end{equation}  
Combinining \eqref{eq:gctcnu} and \eqref{eq:ghalfsemi} we conclude. 
\end{proof}

The result above allows us to increase the regularity in time at the cost of reducing the regularity in space and it will be used to find bounds involving the solution to the regularised Fokker-Planck equation $\rho^N$.

\begin{lemma}\label{lem:holder_reg_der_pde}
For $f \in L^\infty_T C^\xi$ for some $\xi \in \mathbb R$ define $g$ by \eqref{eqn:14}.
Then $g \in L^\infty_T C^{\zeta}$ for any $\zeta < \xi + 2$, and $\|g\|_{L^\infty_T C^\zeta} \leq c_3\| f \|_{L_T^\infty  C^\xi}$, where
\begin{equation*}
c_3 = c_3(T, \xi, \zeta) = c^1_S \frac{T^{1 - \frac{\zeta - \xi}{2}}}{1 - \frac{\zeta - \xi}{2}}.
\end{equation*}
\end{lemma}
\begin{proof}
We only need to derive a bound for $\| g(t) \|_{C^\zeta}$ which is uniform in $t \in [0, T]$. We have
\begin{equation*}
    \begin{split}
      \| g(t) \|_{C^\zeta} 
      &= \left\| \int_0^t P_{t - s} f(s) ds \right\|_{C^\zeta}
      \leq \int_0^t \left\| P_{t - s} f(s) \right\|_{C^\zeta} ds \\
      &\leq \int_0^t c^1_S (t - s)^{-\frac{\zeta - \xi}{2}} \| f \|_{C^\xi} ds
      = c^1_S \frac{t^{1 - \frac{\zeta - \xi}{2}}}{1 - \frac{\zeta - \xi}{2}} \| f \|_{C^\xi}
      \le c^1_S \frac{T^{1 - \frac{\zeta - \xi}{2}}}{1 - \frac{\zeta - \xi}{2}} \| f \|_{C^\xi},
    \end{split}
  \end{equation*}
where $c^1_S$ is the constant from Lemma \ref{lem:schauder}.
\end{proof}

\section{Regularity properties  of \texorpdfstring{$\rho^N$}{rho\textasciicircum N}}\label{sec:rho}

In this section, we derive estimates for $1/2$-H\"older (in time) norm of $\rho^N$ and for the spatial derivative of $\rho^N$, which will be needed in the following sections.

We start by citing known results. The bounds established in \cite[Eqs (26), (27) and (28)]{chaparro2023} via Schauder estimates read, in our setting, as follows: for any $\epsilon > 0$ we have
\begin{align}
  [b^N]_{\frac{1}{2}, L^{\infty}} &\leq \cb N^{\frac{\epsilon + \hat\beta}{2}} \| b \|_{C^{1/2}_T C^{-\hat\beta}}
  \label{eq:snorm_bn}\\
  \|b^N\|_{C_T L^{\infty}} &\leq \cb N^{\frac{\epsilon + \hat\beta}{2}} \| b \|_{C_T C^{-\hat\beta}}
  \label{eq:norm_bn}\\
  \| b^N_x\|_{C_T L^{\infty}} &\leq \cb N^{\frac{\epsilon + \hat\beta + 1}{2}} \| b \|_{C_T C^{-\hat\beta}},
  \label{eq:norm_dbn}
\end{align}
where the constant $\cb$ depends on $\epsilon$, $\hat\beta$ and $T$. In what follows, $\epsilon$ will appear in the main result of the paper and in many intermediate estimates. We do not know the limit of $\cb$ as $\epsilon \to 0$, hence, we cannot set $\epsilon = 0$ in the above estimates. However, $\epsilon$ is meant to be arbitrarily small but fixed.

Below we will obtain bounds on $\rho^N$ which will enable us to derive in the next section analogous estimates as in \eqref{eq:snorm_bn}-\eqref{eq:norm_dbn} but for $F(\rho^N) b^N$.  By Proposition \ref{prop:mildSoln} we know that $\rho$ and $\rho^N$ are continuous in time with values in $L^\infty$, but we also need estimates for the time-H\"older norm, and for the $L^\infty$ norm of the derivative in space with an explicit dependence on $N$.

\begin{lemma}
  \label{lem:HN}
  Let us denote by $H^N = \tilde F (\rho^N) b^N$, and its derivative in space by $H^N_x$, then for $\eta \in (0, 1-\beta)$
  \begin{equation}\label{eqn:25}
      H^N \in C_T C^\eta
      \text{\quad and \quad}
      H^N_x \in C_T C^{\eta - 1}.
  \end{equation}
  Furthermore,
  \begin{equation}\label{eqn:26}
    \|H^N_x\|_{L_T^\infty C^{\eta - 1}}
     \leq
    c_4
    N^{\frac{\eta + \hat\beta}{2}},
  \end{equation}
  where the constant $c_4$ depends on $T, \|\rho_0\|_{C^{\frac12\vee\eta}}, \sup_N \|b^N\|_{C_T C^{-\beta}} $ and on $\|b\|_{C_T C^{-\hat\beta}}$.
\end{lemma}

\begin{proof}
First we show that for each $t\in[0,T]$ we have $H^N (t)  \in C^\eta$ and we bound its norm uniformly in $t$. Recalling that $  \|H^N (t)  \|_{C^\eta} =     \|H^N (t)  \|_{L^\infty}     +     [H^N (t)  ]_\eta$,  we bound the first term as $\|H^N (t)  \|_{L^\infty} \le \|(\tilde F(\rho^N))(t)\|_{L^\infty} \|b^N(t)\|_{L^\infty}$ and the second term as
  \begin{equation}\label{eq:seminorm_FRhob}
   [H^N (t)  ]_\eta =  [(\tilde F(\rho^N) b^N)(t)]_\eta
    \leq
    \|b^N(t)\|_{L^\infty} [\tilde F(\rho^N)(t)]_\eta
    +
    [b^N(t)]_\eta \|\tilde F(\rho^N)(t)\|_{L^\infty}.
  \end{equation}
Thus we have
  \begin{equation}\label{eq:H}
    \begin{aligned}
\|H^N(t)\|_{C^\eta}
      &\leq
      \|(\tilde F(\rho^N))(t)\|_{L^\infty} \|b^N(t)\|_{L^\infty}
      +
      \|b^N(t)\|_{L^\infty} [\tilde F(\rho^N)(t)]_\eta
      +
      [b^N(t)]_\eta \|\tilde F(\rho^N)(t)\|_{L^\infty}
      \\
      &=
      \|b^N(t)\|_{L^\infty}
      (\|(\tilde F(\rho^N))(t)\|_{L^\infty} + [\tilde F(\rho^N)(t)]_\eta)
      +
      [b^N(t)]_\eta \|\tilde F(\rho^N)(t)\|_{L^\infty}
      \\
      &\leq
      \|b^N(t)\|_{L^\infty}
      \|(\tilde F(\rho^N))(t)\|_{C^\eta}
      +
      [b^N(t)]_\eta \|\tilde F(\rho^N)(t)\|_{C^\eta}
      \\
      &=
      \|(\tilde F(\rho^N))(t)\|_{C^\eta}
      \|b^N(t)\|_{C^\eta}.
    \end{aligned}
  \end{equation}
Schauder estimates (Lemma  \ref{lem:schauder}) yield 
\[
 \|b^N(t)\|_{C^\eta} \leq c_S^1 N^{ \frac{\hat\beta + \eta}2 }\|b(t)\|_{C^{-\hat\beta}} \leq  c_S^1 N^{ \frac{\hat\beta + \eta}2 }\|b\|_{C_T C^{-\hat\beta}} .
\]
We apply Lemma \ref{lem:prop_op_F} for $\tilde F$ with $\alpha = \eta$  to obtain 
\[
\|\tilde F (\rho^N(t))\|_{C^\eta} \le 
c_\varphi(1 + \|\rho^N(t)\|_{C^\eta}).
\]
By \cite[Proposition 3.3]{issoglioMcKeanSDEsSingular2023} (choosing  $\alpha=\frac12\vee \eta$, and using that $b^N \in C_TC^{-\beta}$ and $\rho_0 \in C^{\frac12\vee \eta}$) there is an increasing function $\kappa$ (depending also on $T, \|\rho_0\|_{C^{\frac12\vee\eta}}$) such that
  \begin{equation*}
\sup_{t\in[0,T]}\|\rho^N(t)\|_{C^\eta}=   \|\rho^N\|_{C_T C^{\eta}} \leq   \|\rho^N\|_{C_T C^{\frac12\vee \eta}} \leq \kappa \left(\|b^N\|_{C_T C^{-\beta}}\right) \leq \kappa \left( \sup_N \|b^N\|_{C_T C^{-\beta}}\right),
  \end{equation*}
and the latter is finite because $b^N \to b$ in $C_T C^{-\beta}$ as $N \to \infty$.
Putting everything together in \eqref{eq:H} and taking the sup over $t$ we obtain
\[
\|H^N\|_{L_T^\infty C^{\eta }} = \|\tilde F(\rho^N) b^N\|_{L^\infty_T C^\eta}
\leq 
c_\varphi \Big(1 +  \kappa \big( \sup_N \|b^N\|_{C_T C^{-\beta}}\big) \Big)
c_S^1 N^{ \frac{\hat\beta + \eta}2 }\|b\|_{C_T C^{-\hat\beta}}
= C  N^{ \frac{\hat\beta + \eta}2 },
\]
where the constant $C$ depends on the Schauder  constant $c^1_S$, the constant $c_\varphi$, the final time $T$, and on the norms $\|\rho_0\|_{C^{\frac12\vee\eta}}, \sup_N \|b^N\|_{C_T C^{-\beta}} $ and  $\|b\|_{C_T C^{-\hat\beta}}$, but crucially not on $N$.
To obtain \eqref{eqn:26} it is now enough to use  Bernstein inequality (Lemma \ref{lem:bernstein}) which ensures $\|H^N_x\|_{L_T^\infty C^{\eta - 1}} \leq c_B \|H^N\|_{L_T^\infty C^{\eta}}$. This gives \eqref{eqn:26} where $c_4:= C c_B$.
 
 Now let us prove the continuity in $t$ of $H^N$. Take $t,s \in [0, T]$. 
 We know that $H^N(s)$ and $H^N(t)$ belong to $C^\eta$, and with similar computations as in \eqref{eq:H} we have
  \begin{equation}\label{eqn:L1}
  \begin{aligned}
    \| &H^N(t)  - H^N(s)\|_{C^\eta} \\
    &=\| (\tilde F(\rho^N) b^N)(t) - (\tilde F(\rho^N) b^N)(s)\|_{C^\eta}\\
    &\leq
    \| ((\tilde F(\rho^N))(t) - (\tilde F(\rho^N))(s)) b^N(t)\|_{C^\eta}
    +
    \|(\tilde F(\rho^N))(s) (b^N(t) - b^N(s))\|_{C^\eta}\\
    &\leq
    \| (\tilde F(\rho^N))(t) - (\tilde F(\rho^N))(s)\|_{C^\eta}
    \|b^N(t)\|_{C^\eta}
    +
    \|(\tilde F(\rho^N))(s)\|_{C^\eta}
    \|b^N(t) - b^N(s)\|_{C^\eta}.
\end{aligned}
\end{equation}
The second term in \eqref{eqn:L1} tends to zero as $t\to s$ since $b^N\in C_T C^\eta$ and  $ \|(\tilde F(\rho^N))(s)\|_{C^\eta} \leq  \|\tilde F(\rho^N)\|_{L_T^\infty C^\eta} <\infty$ by the computations above.
Let us consider now the first term in \eqref{eqn:L1}. For fixed $t, s$ we know that $\rho^N(t)$ and $\rho^N(s)$ are elements of $C^\eta$ so by Lemma \ref{lem:prop_op_F} we have the bound
\begin{align*}
&\| (\tilde F(\rho^N))(t) - (\tilde F(\rho^N))(s)\|_{C^\eta}\\
&\leq
    c_{\varphi}
    \big(1 + \|\rho^N(t)\|_{C^\eta}^2 + \|\rho^N(s)\|_{C^\eta}^2\big)^{1/2}
    \|\rho^N(t) - \rho^N(s)\|_{C^\eta}\\
&\leq
    c_{\varphi}
    \big(1 + \|\rho^N(t)\|_{C^{\eta\vee \frac12}}^2 + \|\rho^N(s)\|_{C^{\eta\vee \frac12}}^2\big)^{1/2}
    \|\rho^N(t) - \rho^N(s)\|_{C^\eta} \\
&\le c_\varphi (1 + 2 \sup_N \|\rho^N\|_{C_T C^{\eta \vee \frac12}} ) \|\rho^N(t) - \rho^N(s)\|_{C^\eta},
\end{align*}
 where  $\sup_N\|\rho^N\|_{C_T C^{\eta \vee \frac12}}  $ is finite since  $\rho^N \to \rho$ in $C^{\eta \vee \frac12} $ as $\hat\beta< \eta\vee \frac12<1-\hat\beta$. Using that  $\rho^N \in C_TC^\eta$, the right-hand side converges to $0$ as $t \to s$. This proves that $H^N \in C_T C^\eta$.  
 By Bernstein inequality, for every $s,t\in[0,T]$ we have
  \begin{equation*}
  \|H_x^N(t) - H_x^N(s)\|_{C^{\eta-1}}  =  \|\partial_x( H^N(t) - H^N(s))\|_{C^{\eta-1}}  \leq   c_B\|H^N(t) - H^N(s)\|_{ C^\eta},
  \end{equation*}
  therefore the continuity in time of the derivative $H^N_x$  in $C^{\eta-1}$ holds by the  continuity in time of  $H^N$  in $C^{\eta}$.
\end{proof}

Using Lemma \ref{lem:holder_reg_pde}, we will show that $\rho^N$ is $\frac12$-H\"older continuous in $t$.
\begin{lemma}\label{lem:rhoN_time_reg}
We have  $\rho^N \in C^{1/2}_T L^\infty$ and for any $\eta\in(0, (1-\beta) \wedge \nu)$
\begin{equation*}
    [\rho^N]_{1/2, L^\infty} 
    \leq
    c_5 + c_6 N^{\frac{\eta + \hat\beta}{2}},
\end{equation*}
where $c_5 = c^2_S \|\rho_0\|_{C^{1 + \nu}}$ and $c_6 = c_2 c_4$.
\end{lemma}
\begin{proof}
We observe that $\rho^N \in C_TC^\alpha$ for some $\alpha>0$ (see Proposition \ref{prop:mildSoln}), so $\rho^N \in C_T L^\infty$. We only need to prove the H\"older continuity in time, i.e., we need to bound the seminorm $[\rho^N]_{\frac{1}{2}, L^\infty}$. Recall that by \eqref{eq:mild_soln_rfp}, $\rho^N(t) = P_t \rho_0 - \int_0^t P_{t - s} \partial_x H^N(t) ds$. Without any increase in complexity, we will provide a stronger bound: for any $\delta \in (0, \eta)$ we have
  \begin{equation}\label{eqn:31}
    \begin{aligned}
      \relax
      [\rho^N]_{\frac{1}{2}, C^{\eta - \delta}}
      &= \sup_{t \neq s}\frac{\|\rho^N(t) - \rho^N(s)\|_{C^{\eta - \delta}}}{|t - s|^{1/2}}\\
      &= \sup_{t \neq s}\frac{\left\|P_t \rho_0 - P_s \rho_0 + 
      \int_0^t P_{t - u} H^N_x(u) du - \int_0^s P_{s - u} H^N_x(u) du \right\|_{C^{\eta - \delta}}}{|t - s|^{1/2}}\\ 
      &\leq \sup_{t \neq s} \frac{\|P_t \rho_0 - P_s \rho_0\|_{C^{\eta - \delta}}}{|t - s|^{1/2}} + \sup_{t \neq s}\frac{\left\|\int_0^t P_{t - u} H^N_x(u) du - \int_0^s P_{s - u} H^N_x(u) du\right\|_{C^{\eta - \delta}}}{|t - s|^{1/2}}.
    \end{aligned}
\end{equation}
For the first term, we use the Schauder estimates (Lemma \ref{lem:schauder}(ii) with $\zeta = 1/2$):
\[
\|P_t \rho_0 - P_s \rho_0\|_{C^{\eta-\delta}}
    \leq c^2_S|t - s|^{1/2}\| \rho_0 \|_{C^{\eta - \delta + 1}} \leq c^2_S|t - s|^{1/2}\| \rho_0 \|_{C^{1 + \nu}} = c_5 |t - s|^{1/2},
\]
having used that $\eta - \delta + 1 < \eta + 1 < 1 + \nu$.

For the second term of \eqref{eqn:31}, we start from noting that $H^N_x \in C_T C^{\eta-1}$ by Lemma \ref{lem:HN}. An application of Lemma \ref{lem:holder_reg_pde} with $f = H^N_x$ and $\zeta = \eta - \delta$ yields
\[
t \mapsto \int_0^t P_{t - u} H^N_x(u) du \in C^{1/2}_{T} C^{\eta-\delta}
\]
with the norm
\[
\left\|\int_0^\cdot P_{\cdot - u} H^N_x(u) du\right\|_{C^{1/2}_TC^{\eta-\delta}} \le c_2 \| H^N_x\|_{C_TC^{\eta-1}} \le c_2 c_4 N^{\frac{\eta + \hat \beta}2},
\]
where the last inequality is by \eqref{eqn:26}. It remains to note that $[\cdot]_{1/2, L^\infty} \le \|\cdot\|_{C^{1/2}_T L^\infty} \le \|\cdot\|_{C^{1/2}_T C^{\eta-\delta}}$.
\end{proof}

\begin{corollary}
We have  $\rho^N \in C^{1/2}_T C^{\eta}$ for any $\eta\in(0, (1-\beta) \wedge \nu)$ with
\begin{equation}\label{eq:rhoN_seminorm}
    [\rho^N]_{1/2,C^\eta} \leq c_5 + c_6 N^{\frac{\eta + \hat\beta + \delta}{2}},
\end{equation}
where $\delta > 0$ is arbitrary but the constants $c_5, c_6$ from Lemma \ref{lem:rhoN_time_reg} depend on it.
\end{corollary}
\begin{proof}
Assume first that $\delta$ is such that $\eta + \delta < (1-\beta) \wedge \nu$. The estimate of the H\"older seminorm $[\rho^N]_{1/2,C^\eta}$ follows directly from the estimates in the proof of Lemma \ref{lem:rhoN_time_reg} upon taking $\eta = \eta + \delta$. If $\delta$ in the statement does not satisfy the condition $\eta + \delta < (1-\beta) \wedge \nu$ we obtain the result for any $\delta'$ that satisfies it, and then notice that the right-hand side of \eqref{eq:rhoN_seminorm} is increasing in $\delta$, hence it holds for all $\delta \geq \delta'$. The finiteness of the H\"older seminorm and the fact that $[0,T] $ is compact  imply that $\rho^N \in C^{1/2}_T C^\eta$. 
\end{proof}

\begin{lemma}\label{lem:rhoN_der_reg_time}
We have $\rho^N_x \in L^\infty_T L^\infty$ and, for any $\eta \in (0, \nu)$,
\begin{equation*}
    \| \rho_x^N \|_{L_T^\infty L^\infty}
    \leq 
    c_7 + c_8 N^{\frac{\eta + \hat\beta}{2}},
\end{equation*}
where $c_7 = c_7 (\|\rho_0\|_{C_T C^{1 + \nu}})$ and $c_8 = c_3 c_4$ depends on $\eta$, $T, \|\rho_0\|_{C^{\frac12\vee\eta}}, \sup_N \|b^N\|_{C_T C^{-\beta}} $ and on $\|b\|_{C_T C^{-\hat\beta}}$.
\end{lemma}

\begin{proof}
Fix $\eta \in (0, \nu)$ and $\delta \in (0, \eta)$. Recall that by \eqref{eq:mild_soln_rfp}, $\rho^N(t) = P_t \rho_0 - \int_0^t P_{t - s} \partial_x H^N(t) ds$. We will show that $\rho^N \in L^\infty C^{\eta + 1 - \delta}$ by bounding its norm. We have
\begin{equation}\label{eqn:33}
\| \rho^N \|_{L_T^\infty C^{\eta + 1 - \delta}}
\leq\left\| P_\cdot \rho_0 \right\|_{L_T^\infty C^{\eta + 1 - \delta}} + \left\| \int_0^\cdot P_{\cdot-s} H_x^N(s) ds \right\|_{L_T^\infty C^{\eta + 1 - \delta}}.
\end{equation}
The continuity of the operator $P_t$ on $C^{\eta + 1- \delta}$ implies $\|P_t \rho_0\|_{C^{\eta + 1- \delta}} \leq C\|\rho_0\|_{C^{\eta + 1- \delta}}$ for some constant $C>0$, so 
\begin{equation}\label{eq:derbn_1st_term}
\|P_\cdot \rho_0\|_{L_T^\infty C^{\eta + 1- \delta}} \leq C\|\rho_0\|_{C^{\eta + 1- \delta}} \leq C \|\rho_0\|_{C^{1+\nu}},
\end{equation}
where the right-hand side is bounded as $\eta - \delta \le \nu$ and by Assumption \ref{ass:rho0}. For the second term of \eqref{eqn:33}, we apply Lemma \ref{lem:holder_reg_der_pde} with $\zeta = \eta + 1 - \delta$ to obtain
\[
\left\|\int_0^\cdot P_{\cdot - s} H_x^N(s)ds\right\|_{L_T^\infty C^{\eta+1-\delta}}
    \leq c_3 \|H_x^N \|_{L_T^\infty C^{\eta - 1}} \leq c_3 c_4 N^{\frac{\eta + \hat \beta}2},
\]
where we used that $H^N_x \in C_T C^{\eta - 1}$ by Lemma \ref{lem:HN} and the bound $\|H^N_x\|_{C_T C^{\eta - 1}} \le c_4 N^{\frac{\eta + \hat \beta}2}$, where the constant $c_4$ depends on $\eta$ and $\delta$. Inserting the bounds into \eqref{eqn:33} gives
\[
\| \rho^N \|_{L_T^\infty C^{\eta + 1 - \delta}}
\leq 
C \|\rho_0\|_{C^{1+\nu}} + c_3 c_4 N^{\frac{\eta + \hat \beta}2}.
\]
By the Bernstein inequality, Lemma \ref{lem:bernstein}, applied for each time $t$, we have
\[
\| \rho_x^N \|_{L_T^\infty C^{\eta - \delta}} \le c_B \| \rho^N \|_{L_T^\infty C^{\eta + 1 - \delta}}.
\]
As $\eta - \delta > 0$, the bound $\| \rho_x^N \|_{L_T^\infty L^\infty} \le \| \rho_x^N \|_{L_T^\infty C^{\eta - \delta}}$ holds. As we do not want the estimate in the lemma to depend on $\delta$, we need to specify $\delta$ as a function of $\eta$, for example, $\delta = \eta/2$. The constants  $\eta$ and $\delta$ influence the constant $c_3$.
\end{proof}

\section{Three key bounds for \texorpdfstring{$B^N$}{B\textasciicircum N}}\label{sec:three_bounds}

The main aim of this section is to derive regularity results for the smoothed drift $B^N$ in order to apply similar arguments from \cite{chaparro2023}. More precisely, we are aiming at obtaining analogoues of results in \cite[Proposition 4 and Corollary 5]{chaparro2023}. In particular, we need to show that $B^N \in C^{1/2}_T L^\infty \cap L^\infty_T C^1_b$, where, $B^N = F(\rho^N) b^N$. In this section, we fix $\epsilon \in (0, (1-\beta) \wedge \nu)$ and consider the bounds \eqref{eq:snorm_bn}--\eqref{eq:norm_dbn}, where the constant $c_b$ depends on $\epsilon$. We recall that $\epsilon$ is meant to be small but we cannot send it to $0$ as the constant $c_b$ in the aforementioned bounds may explode. We will also apply bounds from the previous section with $\eta = \epsilon$ and the constants there explode when $\eta \to 0$. 

\begin{lemma}\label{lem:BNhalf}
We have $B^N \in C^{1/2}_T L^\infty$ and
\begin{equation*}
\begin{aligned}
&\|B^N\|_{L^\infty_T L^\infty} \leq c_9N^{\frac{\epsilon + \hat\beta}{2}},\\
&[B^N]_{\frac12, L^\infty} \leq c_{10} N^{\frac{\epsilon + \hat\beta}{2}} + c_{11} N^{\epsilon + \hat\beta},
\end{aligned}
\end{equation*}
where
\begin{align*}
c_{9} &= \|F\|_{\infty} c_b \|b\|_{C_T C^{-\hat\beta}},\\
c_{10} &= c_b \|b\|_{C_T C^{-\hat\beta}} \big( \|F\|_\infty + \|F_x\|_{\infty} c_5\big),\\
c_{11} &=  c_b \|b\|_{C_T C^{-\hat\beta}} \|F_x\|_{\infty} c_6.
\end{align*}
\end{lemma}
\begin{proof}
Since $F$ is bounded, we write
\begin{equation*}
\|B^N\|_{L^\infty_T L^\infty} = \|F(\rho^N)b^N\|_{L^\infty_T L^\infty}
\leq 
\|F\|_{\infty} \|b^N\|_{L^\infty_T L^\infty}
\leq
\|F\|_{\infty} c_b N^{\frac{\epsilon + \hat\beta}{2}}\|b\|_{C_T C^{-\hat\beta}},
\end{equation*}
where \eqref{eq:norm_bn} is used in the last inequality. This proves the first bound in the lemma.

To show that $B^N \in C^{1/2}_T L^\infty$, it is sufficient to bound the following seminorm
\[
[B^N]_{1/2, L^\infty} = \sup_{t \neq s}\frac{\|B^N(t, \cdot) - B^N(s, \cdot)\|_{L^\infty}}{|t - s|^{1/2}}.
\]
We write
\begin{equation}\label{eq:regul_BN}
\begin{aligned}
&\| B^N(t, \cdot) - B^N(s, \cdot) \|_{L^\infty}\\
&=
\| F(\rho^N(t, \cdot)) b^N(t, \cdot) -
F(\rho^N(t, \cdot)) b^N(s, \cdot) + F(\rho^N(t, \cdot)) b^N(s, \cdot) - F(\rho^N(s, \cdot)) b^N(s, \cdot) \|_{L^\infty}\\
& \leq
\| F(\rho^N(t, \cdot)) [b^N(t, \cdot) - b^N(s, \cdot)] \|_{L^\infty}
+
\| b^N(s, \cdot) [F(\rho^N(t, \cdot))- F(\rho^N(s, \cdot))] \|_{L^\infty}.
\end{aligned}
\end{equation}
For the first term, the boundedness of $F$ and \eqref{eq:snorm_bn} yield
\begin{align*}
\| F(\rho^N(t, \cdot)) [b^N(t, \cdot) - b^N(s, \cdot)] \|_{L^\infty}
&\le
\|F\|_\infty \|b^N(t, \cdot) - b^N(s, \cdot)\|_{L^\infty}\\
&\le
\|F\|_\infty [b^N]_{1/2, L^\infty} |t-s|^{1/2}\\
&\le
\|F\|_\infty c_b N^{\frac{\epsilon+\hat\beta}2} \|b\|_{C_TC^{-\hat\beta}} |t-s|^{1/2}.
\end{align*}

In order to bound the second term in \eqref{eq:regul_BN}, we recall that the derivative $F_x$ is bounded and $b^N$ is bounded, so
  \begin{equation*}
    \begin{aligned}
&      \| b^N(s, \cdot) [F(\rho^N(t, \cdot))- F(\rho^N(s, \cdot))] \|_{L^\infty}\\
      & \leq \| b^N \|_{C_T L^\infty} \| F(\rho^N(t, \cdot)) - F(\rho^N(s, \cdot)) \|_{L^\infty} \\
      &\leq \| b^N \|_{C_T L^\infty} \|F_x\|_\infty \| \rho^N(t, \cdot) - \rho^N(s, \cdot)\|_{L^\infty}.
    \end{aligned}
  \end{equation*}
Inserting the estimate \eqref{eq:norm_bn} for $\| b^N \|_{C_T L^\infty}$ and the bound from Lemma \ref{lem:rhoN_time_reg} with $\eta = \epsilon$ for $\| \rho^N(t, \cdot) - \rho^N(s, \cdot)\|_{L^\infty}$ gives
\[
\| b^N(s, \cdot) [F(\rho^N(t, \cdot))- F(\rho^N(s, \cdot))] \|_{L^\infty}
\le
\cb N^{\frac{\epsilon + \hat \beta}{2}}\| b \|_{C_T C^{-\hat\beta}} \|F_x\|_\infty
(c_5 + c_6 N^{\frac{\epsilon + \hat\beta}2}) |t-s|^{1/2}.
\]

Putting the above estimates into \eqref{eq:regul_BN} and dividing by $|t-s|^{1/2}$ provides a bound for $[B^N]_{1/2, L^\infty}$ and completes the proof.
\end{proof}

\begin{lemma}\label{lem:BNder}
We have  $B^N \in L^\infty_T C^1_b$ and
  \begin{equation} \label{eq:BNx}
      \|B^N_x\|_{L_T^\infty L^\infty}
      \leq
      c_{12} N^{\frac{\epsilon + \hat\beta}{2}} 
      + c_{13} N^{\epsilon + \hat\beta}
      + c_{14} N^{\frac{\epsilon + \hat\beta + 1}{2}},
  \end{equation}
  where
  \begin{equation*}
    c_{12} =  \|F_x\|_\infty c_7c_b \|b\|_{C_TC^{-\hat\beta}},
  \end{equation*}
  \begin{equation*}
    c_{13} = \|F_x\|_\infty c_8c_b  \|b\|_{C_TC^{-\hat\beta}},
  \end{equation*}
  \begin{equation*}
    c_{14} =  \|F\|_\infty c_b \|b\|_{C_TC^{-\hat\beta}}.
  \end{equation*}
\end{lemma}
\begin{proof}
For any $t\in[0,T]$, it is clear that $B^N(t)$ is differentiable in $x$ because so are $F$ by assumption, $\rho^N(t, \cdot)$ by Lemma \ref{lem:rhoN_der_reg_time} and $b^N(t, \cdot)$ by \eqref{eq:norm_dbn}. To prove \eqref{eq:BNx} we write
  \begin{equation}\label{eq:bound_BN_der}
    \begin{aligned}
      B^{N }_{x }(t, \cdot) 
     = \frac{d }{dx }[F(\rho^{N }(t, \cdot)) b^{N }(t, \cdot)] 
      & = F_x(\rho^{N }(t, \cdot))\rho^{N }_x(t, \cdot) b^{N }(t, \cdot) + F(\rho^{N }(t, \cdot)) b^{N }_{x }(t, \cdot),
    \end{aligned}	
  \end{equation}
 and bound its $L^\infty$-norm as follows
  \begin{equation}\label{eqn:36}
    \begin{aligned}
      \|B^N_x\|_{L_T^\infty L^\infty}
          &\leq \|F_x(\rho^{N })\rho^{N }_x b^{N }\|_{L_T^\infty L^\infty} + \|F(\rho^{N }) b^{N }_{x }\|_{L_T^\infty L^\infty}\\
      &\leq \|F_x(\rho^{N })\|_{L_T^\infty L^\infty}\|\rho^{N }_x\|_{L_T^\infty L^\infty} \|b^{N }\|_{L_T^\infty L^\infty} + \|F(\rho^{N })\|_{L_T^\infty L^\infty} \|b^{N }_{x }\|_{L_T^\infty L^\infty}.
    \end{aligned}	
  \end{equation}
Recalling that $F$ and $F_x$ are bounded by Assumption \ref{ass:FF} and using equations \eqref{eq:norm_bn}, \eqref{eq:norm_dbn} and  Lemma \ref{lem:rhoN_der_reg_time} with $\eta=\epsilon$
 we get
  \begin{align*}
    \|B^{N }_{x}\|_{L_T^\infty L^\infty}
    &\leq
    \|F_x\|_\infty (c_7+ c_8 N^{\frac{\epsilon+ \hat\beta}2}) c_b N^{\frac{\epsilon + \hat \beta}2} \|b\|_{C_TC^{-\hat\beta}}
 + \|F\|_\infty c_b N^{\frac{\epsilon + \hat \beta+1}2} \|b\|_{C_TC^{-\hat\beta}}\\
 &\leq
 \|F_x\|_\infty c_7c_b N^{\frac{\epsilon + \hat \beta}2} \|b\|_{C_TC^{-\hat\beta}} +  \|F_x\|_\infty c_8c_b N^{\epsilon + \hat \beta} \|b\|_{C_TC^{-\hat\beta}} \\
& \quad + \|F\|_\infty c_b N^{\frac{\epsilon + \hat \beta+1}2} \|b\|_{C_TC^{-\hat\beta}},
  \end{align*}
which concludes the proof.
\end{proof}

\section{Convergence rate with an exact solution to PDE}\label{sec:main_result}
This section culminates our analysis of the convergence error of the Euler scheme for \eqref{eq:MVSDEInit} under the assumption that $\rho^N$, which solves the regularised Fokker-Plank PDE \eqref{eq:FPEqnReg}, is known exactly. This is relaxed in Section \ref{sec:rate_approx_PDE}.

Recall \eqref{eq:TriangleIneq}, where the second term disappears as we assume here that the exact solution $\rho^N$ of the PDE \eqref{eq:FPEqnReg} is known. The approximation error is split into two terms: the Euler approximation error for the smoothed solution $X^N$ (Euler error) and the error arising from approximating $X$ by $X^N$ (stability estimate). The estimates that we obtained in the previous sections allow us to employ results from \cite{chaparro2023} to bound those two errors.

For what concerns the Euler error, by Lemma \ref{lem:BNhalf} and \ref{lem:BNder} we have $B^N \in C^{\frac12}_T L^\infty \cap L^\infty C^1_b$. Using \cite[Proposition~4]{chaparro2023} we bound the strong error of the Euler scheme for $X^N$ as follows:
\begin{equation}\label{eqn:nr}
\sup_t \mathbb E|X^{N, m} - X^{N}| \leq d_1(N) m^{-1} + d_2(N) m^{-\frac12},
\end{equation}
where
\begin{align*}
    d_{1}(N) &= c_9 N^{\frac{\epsilon + \hat\beta}{2}} \left( 1 + c_{12} N^{\frac{\epsilon + \hat\beta}{2}} 
      + c_{13} N^{\epsilon + \hat\beta} + c_{14} N^{\frac{\epsilon + \hat\beta + 1}{2}} \right),\\
    d_{2}(N) &= (c_{10} + c_{12}) N^{\frac{\epsilon + \hat\beta}{2}} 
      + (c_{11} + c_{13}) N^{\epsilon + \hat\beta}
      + c_{14} N^{\frac{\epsilon + \hat\beta + 1}{2}},
\end{align*}
and $\epsilon > 0$ is the parameter appearing in \eqref{eq:snorm_bn}--\eqref{eq:norm_dbn}.

Bounding $\sup_t\mathbb E | X^N_{t} - X_{t} |$ is not as quick. We need to adapt \cite[Proposition 6]{chaparro2023} to our setting (Proposition \ref{prop:p6}) and estimate the norm of $B^N - B$ in an appropriate space (Proposition \ref{prop:rr}).

\begin{proposition}\label{prop:p6}
For any $\alpha \in (1/2, 1-\hat\beta)$ and any $\gamma \in (\hat\beta, 1/2)$, there is a constant $c_{15}>0$ that depends on $T$ and $\gamma$  but not on $\alpha$, such that
\begin{equation}\label{eqn:63}
\sup_t\mathbb E | X^N_{t} - X_{t} | \leq c_{15} \|B^N - B\|_{C_T C^{-\gamma}}^{2\alpha - 1}
\end{equation}
for all $N$ sufficient large so that $\|B^N - B\|_{C_TC^{-\gamma}} < 1$. Furthermore, $\E |X^N_t|^2 + \E |X_t|^2 < \infty$ for $t \in [0,T]$. 
\end{proposition}
\begin{proof}
The difference between our setting and that of \cite[Proposition 6]{chaparro2023} is in the initial condition. It is deterministic there and random here. The proof of Proposition 6 in \cite{chaparro2023} takes up its Section 4. The initial condition shows up in several places which we will address now. We will use the notation from \cite{chaparro2023} without recalling it, as the reader needs to consult this paper to understand the arguments below.

We define the process $Y^N_t = \psi^N(t, X^N_t)$ corresponding to the solution $X^N$ in an analogous way as $Y$ corresponds to $X$ via the mapping $\psi$ in the proof of Theorem \ref{thm:strong}. Using the same arguments, we obtain that $\E (Y^N_t)^2 < \infty$ for all $t \in [0, T]$. The integrability is necessary for the arguments in \cite[Section~4]{chaparro2023}. The only place where the initial condition enters meaningfully is the proof of \cite[Prop.~6]{chaparro2023}. There one computes
\begin{equation*}
	\begin{aligned}
		\mathbbm{E}\left[|Y^N_t - Y_t|\right]
		&= \mathbbm{E}|Y^N_0 - Y_0| + \mathbbm{E}\left[\frac{1}{2} L^0_t (Y^N - Y)\right]
		\\
		&\hspace{11pt}+ \lambda \mathbbm{E}\left[\int_0^t \sgn (Y^N - Y)(u^N(s, \psi^N(s, Y^N_s)) - u(s, \psi(s, Y_s))) ds\right].
	\end{aligned}
\end{equation*}
Only the first term needs to be treated differently than with a deterministic initial condition. We have
\begin{align*}
Y^N_0 - Y_0 &= \psi^N(0, X_0) - \psi(0, X_0) = (X_0 + u^N(0, X_0)) - (X_0 + u(0, X_0))\\
&\le \|u^N(0) - u(0)\|_{\infty} \le \|u^N - u\|_{C_TC^{1+\alpha}} \le o \Big(\|B^N - B\|_{C^TC^{-\gamma}}^{2\alpha-1}\Big),
\end{align*}
where the last inequality is justified in the same way as in \cite{chaparro2023} referring to their Lemma 11 and the assumption $\|B^N - B\|_{C_TC^{-\gamma}} < 1$. The remaining arguments of \cite[Section~4]{chaparro2023} required for the proof of our proposition  remain valid.
\end{proof}

\begin{proposition}[Stability estimate]\label{prop:rr}
For any $\alpha \in (1/2, 1-\hat\beta)$ and any $\gamma \in (\hat\beta, 1/2)$, we have
  \begin{equation}
    \sup_t\mathbb E | X^N_{t} - X_{t} | \leq c_{15} \left(c_{16} N^{-\frac{\gamma - \beta}{2}} \right)^{2 \alpha - 1},  
  \end{equation}
for $N$ sufficient large so that $\|B^N - B\|_{C_TC^{-\gamma}} < 1$, where the constant $c_{15}$ does not depend on $\alpha$ and the constant $c_{16}$ is explicitly given by
\begin{align*}
c_{16} = c_\varphi\big(1 + \sup_N \|\rho^N\|_{C_T C^{1/2}} + c_l \|b\|_{C_T C^{-\gamma}} (1 + \sup_N \|\rho^N\|_{C_T C^{1/2}}^2 + \|\rho\|_{C_T C^{1/2}}^2)^{1/2}\big) c^2_S \|b\|_{C_T C^{-\gamma}},
\end{align*}
 where the constant $c_l$ depends only on $\hat\gamma \in (\gamma, 1) \cap (\beta, 1-\beta)$ fixed in the proof.
\end{proposition}

\begin{proof}
Let  $\hat\gamma \in (\gamma, 1)  \cap (\beta, 1-\beta)$. We bound the right-hand side of \eqref{eqn:63} as follows:
  \begin{equation*}
    \begin{aligned}
      \|B^N - B\|_{C_T C^{-\gamma}} 
    &= \|F(\rho^N) b^N - F(\rho) b\|_{C_T C^{-\gamma}}\\
    &= \|F(\rho^N) b^N - F(\rho^N) b + F(\rho^N) b - F(\rho) b\|_{C_T C^{-\gamma}}\\
    &\leq \|F(\rho^N) b^N - F(\rho^N) b\|_{C_T C^{-\gamma}} + \|F(\rho^N) b - F(\rho) b\|_{C_T C^{-\gamma}}\\
    &\leq \|F(\rho^N)\|_{C_T C^{\hat\gamma}} \|b^N - b\|_{C_T C^{-\gamma}} + \|b\|_{C_T C^{-\gamma}}\|F(\rho^N) - F(\rho)\|_{C_T C^{\hat\gamma}},
    \end{aligned}
  \end{equation*}
where in the last inequality we use the property of the point-wise product between a function and a distribution defined using Bony's paraproduct \cite{bony}, see also \cite[Eq.~\lbracket2.9)]{issoglioMcKeanSDEsSingular2023} for the specific application of Bony's paraproduct to the current setting. Here it is crucial that $\hat\gamma \in( \gamma, \infty)$.

The term $\|F(\rho^N)\|_{C_T C^{\hat\gamma}}$ is bounded using Lemma \ref{lem:prop_op_F}
\begin{equation*}
\|F(\rho^N)\|_{C_T C^{\hat\gamma}} 
\leq c_\varphi(1 + \|\rho^{N}\|_{C_T C^{\hat\gamma}})
\leq c_\varphi(1 + \sup_N \|\rho^{N}\|_{C_T C^{\hat\gamma}}),
\end{equation*}
where $\sup_N \|\rho^{N}\|_{C_T C^{\hat\gamma}} < \infty$ by the convergence of $\rho^N \to \rho$ in $C_T C^{\hat\gamma}$ given in \eqref{eq:rhoNminrho} thanks to the choice $\hat\gamma\in(\beta,1-\beta)$. For bounding the term $\|F(\rho^N) - F(\rho)\|_{C_T C^{\hat\gamma}}$ we  use Lemma~\ref{lem:prop_op_F} again
  \begin{equation*}
    \begin{aligned}
      \|F(\rho^N) - F(\rho)\|_{C_T C^{\hat\gamma}} 
    &\leq 
    c_\varphi \big(1 + \|\rho^N\|_{C_T C^{\hat\gamma}}^2 + \|\rho\|_{C_T C^{\hat\gamma}}^2\big)^{1/2} 
    \|\rho^N - \rho\|_{C_T C^{\hat\gamma}}\\
    &\leq c_\varphi \big(1 + \sup_N \|\rho^N\|_{C_T C^{\hat\gamma}}^2 + \|\rho\|_{C_T C^{\hat\gamma}}^2\big)^{1/2} \|\rho^N - \rho\|_{C_T C^{\hat\gamma}}.
    \end{aligned}
  \end{equation*}
By \cite[Proposition 4.4]{issoglioMcKeanSDEsSingular2023}, we obtain $\|\rho^N - \rho\|_{C_T C^{\hat \gamma}} \le  c_l \|b^N - b\|_{C_T C^{-\beta}}$,
where the constant $c_l$ depends on $\hat\gamma$, $\|\rho_0\|_{C^{\hat\gamma}} \le \|\rho_0\|_{C^{1+\nu}}$ and on $\sup_{N} \|b^N\|_{C_TC^{-\beta}}$, the latter being finite as $b^N \to b$ in $C_TC^{-\beta}$. 

Combining the above estimates gives
\begin{align*}
&\|B^N - B\|_{C_T C^{-\beta}}\\
&\leq
c_\varphi\big(1 + \sup_N \|\rho^N\|_{C_T C^{\hat\gamma}} + c_l \|b\|_{C_T C^{-\gamma}} (1 + \sup_N \|\rho^N\|_{C_T C^{\hat\gamma}}^2 + \|\rho\|_{C_T C^{\hat\gamma}}^2)^{1/2}\big) \|b^N - b\|_{C_T C^{-\beta}}.
\end{align*}
Finally, by the Schauder estimates (Lemma \ref{lem:schauder}(ii)) we have
\[
\|b^N - b\|_{C_T C^{-\beta}}
\le
c^2_S \left(\frac{1}{N}\right)^{\frac{\gamma - \beta}{2}} \|b\|_{C_T C^{-\gamma}},
\]
and putting everything together we conclude the proof. As the choice of $\hat \gamma$ does not affect the rate but only the constants, we choose $\hat\gamma=1/2$.
\end{proof}

The final result bounding the strong error of our numerical scheme is obtained by combining the bound \eqref{eqn:nr} and Proposition \ref{prop:rr}.  

\begin{proposition}\label{prop:diff_num_rsol}
  Assume that the conditions for Proposition \ref{prop:rr} hold, then the numerical error is bounded by
  \begin{equation}\label{eq:numerical}
   \sup_{t \in [0, T]} \mathbb E |X^{N,m}_{t} - X_{t}|
    \leq
    d_1(N) m^{-1} + d_2(N) m^{-\frac12}
    +
    c_{15} \left(c_{16} N^{-\frac{\gamma-\beta}{2}}\right)^{2\alpha - 1},
  \end{equation}
  where $d_1(N) , d_2(N) $ come from \eqref{eqn:nr} and  the constants $c_{15}$ and $c_{16}$ come from Proposition \ref{prop:rr} and are independent of $\alpha \in (1/2, 1-\hat\beta)$ but depend on $\gamma \in (\hat\beta, 1/2)$.
\end{proposition}

The above technical proposition expresses the bound for the approximation error in terms of smoothing parameter $N$, the time step $\frac1m$ and a number of auxiliary parameters $\epsilon$, $\alpha$ and $\gamma$. Our aim is to set those parameters and determine how $N$ should depend on $m$ in such a way to obtain an error bound in terms of $\beta$ and $m$. First, we set
\[
\epsilon:= \beta-\hat\beta,
\]
which appears in the estimates in the previous sections starting from \eqref{eq:snorm_bn}--\eqref{eq:norm_dbn} and in $d_1(N), d_2(N)$ in the first two terms of \eqref{eq:numerical}. We note that $\epsilon$ must be smaller than $\nu \wedge (1-\beta)$, see the beginning of Section \ref{sec:three_bounds}, which is guaranteed by Assumption \ref{ass:smallb}, see the discussion in Remark \ref{rem:beta}.

Inserting the above $\epsilon$ into the formulas for $d_1(N)$ and $d_2(N)$, they simplify to
\begin{align*}
    d_{1}(N) &= c_9 N^{\frac{\beta}{2}} \left( 1 + c_{12} N^{\frac{\beta}{2}} 
      + c_{13} N^{\beta} + c_{14} N^{\frac{\beta + 1}{2}} \right),\\ 
    d_{2}(N) &= (c_{10} + c_{12}) N^{\frac{\beta}{2}} 
      + (c_{11} + c_{13}) N^{\beta}
      + c_{14} N^{\frac{\beta + 1}{2}}.
\end{align*}
The fastest growing terms in  $d_1(N)$ and $d_2(N)$ are, respectively, $c_9 c_{14} N^{\beta + \frac12}$ and $c_{14} N^{\frac \beta 2 + \frac12}$. 

We now proceed to determine the relationship between $N$ and $m$ to balance the error of Euler-Maruyama scheme for $X^N$ and the error of approximating $X$ by $X^N$. We postulate that $N=m^{\kappa}$ for some $\kappa>0$ and choose it so as to maximise the convergence rate. The asymptotic behaviour of the first two terms in the bound \eqref{eq:numerical} is guided by the fastest growing terms:
\begin{align}\label{eq:d1d2}
\nonumber
& c_9 c_{14} N^{\beta + \frac12} m^{-1} + c_{14} N^{\frac{ \beta + 1}2} m^{-\frac12} \\
\nonumber
&= c_9 c_{14} m^{\kappa (\beta + \frac12)} m^{-1} + c_{14} m^{\kappa \frac{ \beta + 1}2} m^{-\frac12} \\
&=c_9 c_{14} m^{\kappa (\beta + \frac12) -1} + c_{14} m^{\kappa \frac{ \beta + 1}2 -\frac12}.
\end{align}
For the third term of \eqref{eq:numerical}, we need to choose $\alpha$ and $\gamma$. Written in terms of $m$ and discarding constants, this term takes the form
\[
m^{-\frac{\kappa(\gamma - \beta)(2\alpha - 1)}{2}}.
\]
The rate is optimised by the choice of the largest possible $\gamma$ and $\alpha$. They must satisfy $\gamma < 1/2$ and $\alpha < 1-\hat\beta$, and the upper bounds cannot be reached as related constants might explode. Instead, we pick a small $\lambda > 0$ and set $\alpha = 1 - \beta-\lambda$ and $\gamma = 1/2 - \lambda$. Note that $1-\beta < 1 - \hat \beta$ due to Assumption \ref{ass:smallb}, so $\alpha < 1-\hat\beta$. This leads to
\[
m^{-\frac{\kappa(\gamma - \beta)(2\alpha - 1)}{2}} 
= m^{-\frac{\kappa(1/2 - \lambda - \beta)(1 - 2\beta - 2 \lambda)}{2}} 
= m^{-\kappa(1/2 - \beta - \lambda)^2}.
\]

Finally, let us find the optimal choice for $\kappa$. In order to guarantee the convergence we must have that the exponents in \eqref{eq:d1d2} are negative, that is 
\[
\kappa (\beta + \frac12) -1 <0, \qquad \text{and} \qquad \kappa \frac{ \beta + 1}2 -\frac12 <0,
\]
which simplifies to $\kappa<\frac1{\beta+\frac12}$ and $\kappa< \frac1{\beta+1}$. The second constraint is stronger, so we optimise over $\kappa \in (0, \frac1{\beta+1})$. Then the asymtotic behaviour of \eqref{eq:d1d2} is governed by 
$$
m^{\kappa \frac{ \beta + 1}2 -\frac12};
$$
indeed, this is the leading term as long as $\kappa < 1/\beta$.

In summary, the asymptotic behaviour of \eqref{eq:numerical} is governed by the sum
\[
C_1 m^{\kappa \frac{ \beta + 1}2 -\frac12} + C_2 m^{-\kappa(1/2 - \beta - \lambda)^2}
\]
for some constants $C_1, C_2 > 0$. As the first term is increasing in $\kappa$ while the second term is decreasing in $\kappa$, the sum decreases fastest as $m \to \infty$ if the two exponents are equal:
\[
\kappa \frac{ \beta + 1}2 -\frac12 = -\kappa(1/2 - \beta - \lambda)^2
\]
yielding $\kappa = \frac{1}{(1 + \beta) + 2(1/2 - \beta - \lambda)^2}$. 
We summarise this result in the following theorem. 

\begin{theorem}[Convergence rate]\label{thm:conv}
For any $\lambda \in (0, 1/2 - \beta)$, setting $N(m)=m^\kappa$, where
\[
\kappa= \frac 1{(1+\beta)  + 2 (\frac12 - \beta - \lambda)^2},
\]
yields the convergence rate
\[
\sup_{t \in [0, T]} \mathbb E |X^{N(m),m}_{t} - X_{t}| \leq c  m^{-\frac12 + \frac {(1+\beta)/2}{(1+\beta)  + 2 (\frac12 - \beta - \lambda)^2} }
\]
for some constant $c > 0$ and any $m$ sufficiently large so that $\|B^{N(m)} - B\|_{C_TC^{-(1/2 - \lambda)}} < 1$.
\end{theorem}

\section{Convergence rate with an approximate solution to PDE}\label{sec:rate_approx_PDE}

Our analysis so far assumed that the solution of the PDE for $\rho^N$ is known exactly. This allowed us to distill the probabilistic aspects of the numerical simulation and derive the bound for the approximation error. However, in practice $\rho^N$ must be computed numerically and we will show how this numerical approximation must be done in order to preserve the convergence rate from Theorem \ref{thm:conv} with the same choice of $N(m)$. In this theorem, $m$ stands for the cost of computing one trajectory of the approximate solution of SDE \eqref{eq:MVSDEInit}; indeed, the real computing cost is proportional $m$. The cost of generating multiple trajectories is an appropriate multiplicity of the cost of generating one trajectory. However, the calculation of $\rho^N$ is done only once, so the cost of it is independent from the number of trajectories of $(X_t)$ that are subsequently simulated and it cannot be incorporated into the bound stated in terms of $m$.

Denote by $\hat\rho^N$ a numerical approximation of $\rho^N$, by $\hat X^N$ the solution to the SDE \eqref{eq:MVSDEReg} with $B^N$ replaced by $\hat B^N(x) := F(\hat\rho(x)) b^N(x)$ and by $\hat X^{N, m}$ the Euler approximation of $\hat X^N$. In the following corollary we formulate conditions on the PDE solution $\hat\rho^N$ so that the bound \eqref{eq:numerical} holds with different constants but the same exponents for $m$ and $N$. This will allow us to conclude with the same convergence rate as in Theorem \ref{thm:conv} as the derivation preceding this theorem was purely based on the exponents of $N$ and $m$ in \eqref{eq:numerical}.

\begin{theorem}\label{cor:conv}
Choose $\gamma \in (\hat\beta, 1/2)$ as in Proposition \ref{prop:rr} and $\hat\gamma \in (\gamma, 1) \cap (\beta, 1-\beta)$ as in its proof. Assume that a numerical solver for $\rho^N$ is available with the approximation $\hat\rho^N$ satisfying the following properties: there is a constant $C>0$ such that
\begin{itemize}
\item (property preservation) 
 \begin{align}
  &[\hat \rho^N]_{1/2, L^\infty} \le C [\rho^N]_{1/2, L^\infty}, \label{eqn:p1}\\
  &\|\hat\rho^N_x\|_{L^\infty_T L^\infty} \le C \|\rho^N_x\|_{L^\infty_T L^\infty}, \label{eqn:p2}\\
  &\|\hat\rho^N\|_{C_T C^{\hat\gamma}} \le C \|\rho^N\|_{C_T C^{\hat\gamma}},\label{eqn:p3}
 \end{align}
\item (convergence)
\begin{equation}\label{eqn:p4}
 \|\hat\rho^N - \rho^N\|_{C_TC^{\hat\gamma}} \le C N^{-\frac{\gamma - \beta}{2}}.
\end{equation}
\end{itemize}
Taking $\lambda$ and $N(m) = m^\kappa$ as in Theorem \ref{thm:conv}, we have
\[
\sup_{t \in [0, T]} \E |\hat X^{N(m),m}_{t} - X_{t}| \leq c  m^{-\frac12 + \frac {(1+\beta)/2}{(1+\beta)  + 2 (\frac12 - \beta - \lambda)^2} }
\]
for some constant $c > 0$.
\end{theorem}

\begin{remark}
In the proof of Proposition \ref{prop:rr} we chose arbitrarily $\hat\gamma=1/2$ as this choice had no effect on the convergence rate, only modifying constants. However, the added flexibility of choosing $\hat\gamma$ in Theorem \ref{cor:conv} allows one to use the most convenient value in order to establish the bounds \eqref{eqn:p3} and \eqref{eqn:p4}. 
\end{remark}

\begin{proof}[Proof of Theorem \ref{cor:conv}]
We decompose the approximation error as
\begin{equation}\label{eqn:d1}
\E |\hat X^{N,m}_{t} - X_{t}| \le \E |\hat X^{N,m}_{t} - \hat X^N_{t}| + \E |\hat X^{N}_{t} - X^N_{t}| + \E |X^{N}_{t} - X_{t}|.
\end{equation}
The first term is analogous as $E |X^{N,m}_{t} - X^N_{t}|$ but with $\rho^N$ replaced by $\hat\rho^N$.  The second term is new and its bound will employ the convergence assumption. The last term is bounded in Proposition \ref{prop:rr}.

We will show that the first term of \eqref{eqn:d1} enjoys a bound of the form \eqref{eqn:nr} with the same exponents of $N$ but with different constants. Using  \cite[Proposition 4]{chaparro2023} we have
\[
\sup_t \mathbb E|\hat X^{N, m} - \hat X^{N}| \leq \hat d_1(N) m^{-1} + \hat d_2(N) m^{-\frac12}
\]
with
\[
\hat d_1(N) = \|\hat B^N\|_{L^\infty_T L^\infty}\big(1 + \|\hat B^N_x\|_{L^\infty_T L^\infty} \big), \qquad \hat d_2(N) = \|\hat B^N_x\|_{L^\infty_T L^\infty} + [\hat B^N]_{\frac12, L^\infty}.
\]
We only need to prove that there is a constant $\hat C > 0$ such that $\hat d_1(N) \le \hat C d_1(N)$ and $\hat d_2 (N) \le \hat C d_2(N)$, that is, we have to find analogous bounds  as in Lemmata \ref{lem:BNhalf} and \ref{lem:BNder} for the three terms  $\|\hat B^N\|_{L^\infty_T L^\infty}, [\hat B^N]_{\frac12, L^\infty}$ and $\|\hat B^N_x\|_{L_T^\infty L^\infty}$.  We first notice that the estimate for $\|\hat B^N\|_{L^\infty_T L^\infty}$ in Lemma \ref{lem:BNhalf} uses only the bounedness of $F$ and does not depend on its argument, hence $\|\hat B^N\|_{L^\infty_T L^\infty}= \| B^N\|_{L^\infty_T L^\infty}$. 
For the second term, let us notice that in the arguments in the proof of Lemma \ref{lem:BNhalf}, the only difference arising from the estimation of $[\hat B^N]_{\frac12, L^\infty}$ instead of $[B^N]_{\frac12, L^\infty}$ is in 
\begin{align*}
\| b^N(s, \cdot) [F(\hat\rho^N(t, \cdot))- F(\hat\rho^N(s, \cdot))] \|_{L^\infty}
&\leq \| b^N \|_{C_T L^\infty} \|F_x\|_\infty \| \hat \rho^N(t, \cdot) - \hat \rho^N(s, \cdot)\|_{L^\infty}\\
&\leq \| b^N \|_{C_T L^\infty} \|F_x\|_\infty [\hat \rho^N]_{\frac12, L^\infty} |t-s|^{1/2} \\
&\leq \| b^N \|_{C_T L^\infty} \|F_x\|_\infty C [\rho^N]_{\frac12, L^\infty} |t-s|^{1/2},
\end{align*}
where the last inequality is by \eqref{eqn:p1}. This leads to
\[
[\hat B^N]_{\frac12, L^\infty} \le c_{10} N^{\frac{\epsilon + \hat \beta}{2}} + C c_{11} N^{\epsilon + \hat\beta}
\]
in the terminology of Lemma \ref{lem:BNhalf} with the difference being the constant $c_{11}$ multiplied by $C$. 
The final norm to bound is $\|\hat B^N_x\|_{L^\infty_T L^\infty}$ and we refer to the proof of Lemma \ref{lem:BNder}. In an analogue of \eqref{eqn:36} written for $\hat B^N$, we bound $\|\hat\rho^N_x\|_{L^\infty_T L^\infty}$ by $C \|\rho^N_x\|_{L^\infty_T L^\infty}$ as per \eqref{eqn:p2} to obtain
\[
\|\hat B^N_x\|_{L_T^\infty L^\infty}
\leq \|F_x(\hat \rho^{N })\|_{L_T^\infty L^\infty} C \|\rho^{N }_x\|_{L_T^\infty L^\infty} \|b^{N }\|_{L_T^\infty L^\infty} + \|F(\hat \rho^{N })\|_{L_T^\infty L^\infty} \|b^{N }_{x }\|_{L_T^\infty L^\infty}.
\]
In the proof of Lemma \ref{lem:BNder}, $\|F_x(\hat \rho^{N })\|_{L_T^\infty L^\infty}$ is bounded by $\|F_x\|_\infty$ and the remaining terms do not involve $\hat\rho$. Hence, $\|\hat B^N_x\|_{L^\infty_T L^\infty}$ is bounded analogously as \eqref{eq:BNx}
\[
      \|\hat B^N_x\|_{L_T^\infty L^\infty}
      \leq
      C c_{12} N^{\frac{\epsilon + \hat\beta}{2}} 
      + C c_{13} N^{\epsilon + \hat\beta}
      + c_{14} N^{\frac{\epsilon + \hat\beta + 1}{2}},
\]
where the constant $C$ multiplies the first and second term compared to \eqref{eq:BNx}. We therefore conclude that there is a constant $\hat C$ such that $\hat d_1(N) \le \hat C d_1(N)$ and $\hat d_2 (N) \le \hat C d_2(N)$.

For the second term of \eqref{eqn:d1} we apply an analogue of the bound \eqref{eqn:63} using the notation and arguments from the proof of Proposition \ref{prop:rr}:
\[
\sup_t \E|\hat X^{N}_{t} - X^N_{t}| \le c_{15} \|\hat B^N - B^N\|^{2\alpha-1}_{C_TC^{-\gamma}},
\]
where
\[
\|\hat B^N - B^N\|_{C_TC^{-\gamma}} = \|F(\hat \rho^N) b^N - F(\rho^N) b^N\|_{C_TC^{-\gamma}}
\le \|b^N\|_{C_T C^{-\gamma}}\|F(\hat \rho^N) - F(\rho^N)\|_{C_TC^{\hat\gamma}}.
\]
Further, we have
\begin{align*}
\|F(\hat \rho^N) - F(\rho^N)\|_{C_TC^{\hat\gamma}} 
&\le c_\varphi \big(1 + \sup_N \|\hat\rho^N\|^2_{C_TC^{\hat\gamma}} + \sup_N \|\rho^N\|^2_{C_TC^{\hat\gamma}}\big)^{1/2} \|\hat\rho^N - \rho^N\|_{C_TC^{\hat\gamma}}\\
&\le c_\varphi \big(1 + (C+1)\sup_N \|\rho^N\|^2_{C_TC^{\hat\gamma}}\big)^{1/2} \|\hat\rho^N - \rho^N\|_{C_TC^{\hat\gamma}}\\
&\le c_\varphi \big(1 + (C+1)\sup_N \|\rho^N\|^2_{C_TC^{\hat\gamma}}\big)^{1/2} C N^{-\frac{\gamma - \beta}{2}},
\end{align*}
where the second inequality follows from \eqref{eqn:p3} and the last inequality is from \eqref{eqn:p4}.

In conclusion we obtain an analogue of \eqref{eq:numerical} of the form
\begin{align*}
&\sup_{t \in [0, T]} \mathbb E |\hat X^{N,m}_{t} - X_{t}|\\
&\leq
    \hat C d_1(N) m^{-1} + \hat C d_2(N) m^{-\frac12}
    +
    c_{15} \left(\hat c_{16} N^{-\frac{\gamma-\beta}{2}}\right)^{2\alpha - 1}
    +
    c_{15} \left(c_{16} N^{-\frac{\gamma-\beta}{2}}\right)^{2\alpha - 1},
\end{align*}
where $\hat c_{16} = c_\varphi \big(1 + (C+1)\sup_N \|\rho^N\|^2_{C_TC^{\hat\gamma}}\big)^{1/2} C$ depends on the choice of $\hat \gamma$. This bound, up to constants, is identical to \eqref{eq:numerical}, so the arguments preceding Theorem \ref{thm:conv} apply without any changes and the thesis of Theorem \ref{thm:conv} holds.
\end{proof}

\section{Numerical illustration} \label{sec:numerics}
In this section  we show that the proposed numerical scheme is implementable and produces an empirical convergence rate which is consistent with the theoretical one  (albeit  better). We quantify the empirical convergence rate across different $b$ and $F$ in order to observe the behaviour of the solution $X$ for different cases of the drift $F(\rho)b$. Finally we check whether the scheme provides approximations which are consistent with the theory, by comparing the empirical density of the Euler-Maruyama solution of $\hat X^N$ with the density $\hat \rho^N$ obtained by a numerical solution of the Fokker-Planck PDE.

For the implementation we need to choose an example of distribution $b$. We restrict to the time-homogeneous case for the numerical implementation. 
The typical example of element $b$ in $C^{(-\beta)+}$ is obtained as the derivative of a non-differentiable function, more precisely,  if we consider functions $h \in C^{(-\beta + 1)+}$ then formally their derivative $h'$ is in  $ C^{(-\beta)+}$, so $h'$ plays the role of $b$. Since $b$ is a Schwartz distribution, we are unable to evaluate it directly and require a suitable approximation for it, which is obtained by convoluting $b$ with  a heat kernel $p_{1/N}$. In particular, we have the approximation $b^N$ given by $ h' \ast p_{1/N}$. Notice that  the implementation of $b^N$ on a machine is still not possible because $ h'$ is a distribution. To overcome this problem, we use the commutativity of convolution with respect to a derivative in our favour, namely we use the equality  $b^N= h' \ast p_{1/N} = h \ast p'_{1/N}$, to obtain an operative definition of $b^N$ since both $ h$ and $p'_{1/N}$ are functions. Now for the choice of $h$, it needs to be an element of $ C^{\gamma+}$ for some $\gamma\in (1/2,1)$, and possibly not better than that.  We thus choose a path of a fractional Brownian motion (fBm) with Hurst index $H\in (1/2,1)$. 
 We recommend the reader to see  \cite[ Section 5]{chaparro2023}  for more details.

We use the Python package \texttt{py-pde} \cite{zwickerPypdePythonPackage2020} to solve the Fokker-Planck equation. The solver we use is based on a Runge-Kutta method of order 4(5) for a system of ODEs, which is adapted to solve PDEs.
We choose $T=1$ for simplicity and solve the PDE only once,  with $2^{11}$ time steps and $4 \times 10^3$ points in space, and interpolate the solution to produce a function $\hat\rho^N$ defined on $[0,1]\times \mathbb R$.  The number of time steps is the same as used to compute a proxy for the exact solution.

The numerical solution $\hat \rho^N$ to the Fokker-Planck equation is used in place of the law density of the unknown $X^N$.  Indeed, in \cite{issoglioMcKeanSDEsSingular2023}, the authors prove that the law density of  $X^N$ is the solution to the associated Fokker-Planck PDE \eqref{eq:FPEqnReg}. Thus the McKean-Vlasov equation reduces to an SDE with known drift given by $B^N = F(\hat\rho^N)b^N$. 
\begin{figure}[htbp]
\figtitle{Effect of different $F$ on the drift $F(\rho^N) b^N$}
\begin{minipage}{0.49\textwidth}
    \centering
    \includegraphics[width=0.99\textwidth]{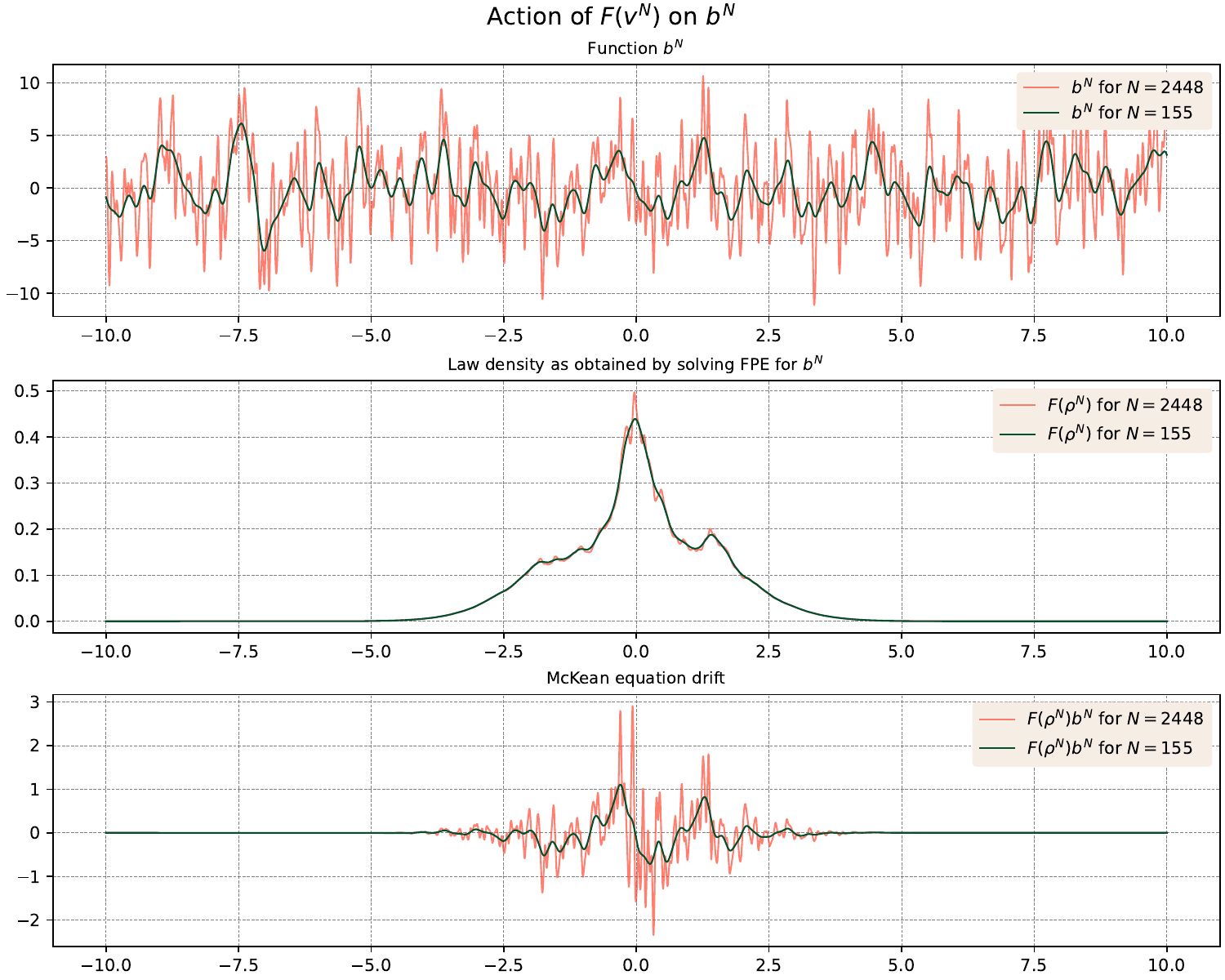}
\end{minipage}
\begin{minipage}{0.49\textwidth}
    \centering
    \includegraphics[width=0.99\textwidth]{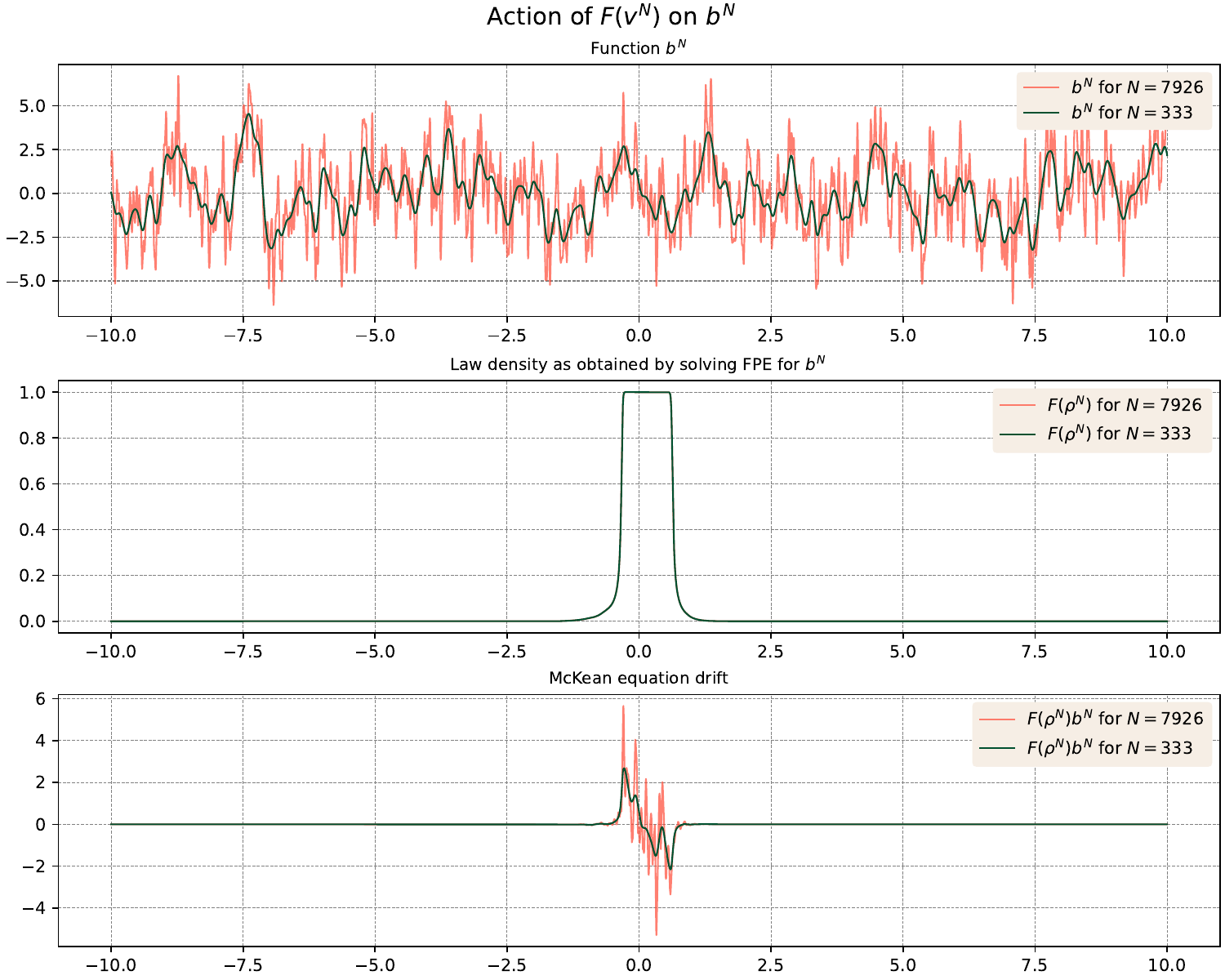}
\end{minipage}    
     \caption{Left panel:  $F(x) = \sin(x)$. Right panel: $F(x) = 1/(1 + \exp(-100(x - 0.2)))$.\\
    Top: the plot of $b^N$ for two different values of $N$ 
    (green line for a `small' $N$, red line for a `large' $N$) for $b\in C^{-\beta}$ with $\beta=0.49$.
   Middle: the plot of $F(\rho^N)$ for two different values of $N$. The  approximation of $\rho^N$ is obtained solving the PDE on a grid with $2^{11}$ time steps and $4\times10^3$ points in $x$.
Bottom: the plot of   the product  $F(\rho^N) b^N$ for two different values of $N$. }
    \label{fig:drift_mckean1}
\end{figure}
In Figure \ref{fig:mckean_rates1} we plot $B^N$ (bottom pictures) for two different values of $N$ (green and red lines) and two different values of $F$ (left and right  panels). The drift $b$ is in $ C^{-\beta}$ with $\beta= 0.49$ and its approximation $b^N$ is shown at the top, while the term $F(\rho^N)$ is plotted  in the middle. The choice of the function $F$ on the left panel is  made to obtain something similar to the identity (at least around zero), while the choice of $F$ on the right panel is an approximation of an indicator function. Notice that the drift $B^N$ is nearly zero (or zero in numerical computations) outside of a bounded interval due to the fact that the main mass of $\hat \rho^N$ is concentrated on a bounded interval (see also Figure \ref{fig:densities_mckean1}). This is unlike $b^N$ which is supported on the whole real line. The overall roughness of the drift $B^N$ is the same as that of $b^N$ since $F(\rho^N)$ is smooth. However, the fact that the drift has practically a compact support makes the SDE very easy to solve numerically whenever the path goes outside the support of $B^N$. We think this may partially explain why we observe that convergence rates for McKean SDEs are slightly higher than those for non-McKean SDEs, c.f. \cite{chaparro2023} (formally equivalent to $F\equiv 1$), even though their drifts $B^N$ and $b^N$ have the same regularity.

Once we have the drift $B^N = F(\hat\rho^N)b^N$, the implementation of the Euler scheme is the one presented for the linear SDE in \cite{chaparro2023}, with the only other difference regarding the initial condition.  Indeed, we cannot use a Dirac delta for the McKean-Vlasov equation, but instead we have to  choose a law which has a suitable density.  For simplicity we stick to a standard normal distribution.
We solve the SDE and store only the values at the terminal time since these are what we will use to calculate the error of the scheme. We use $10^4$ sample paths.
We make a costly computation ($m=2^{11}$ time steps) to get the finest approximation to the SDE, which we treat as the real solution and compare it to each of the coarser approximations ($m= 2^7, 2^8, 2^9$) obtained with less time steps.
We compare the values to calculate the strong error and then compute a linear regression using the $\log$ base 10 of the number of time steps as the explanatory variable and the error as the response, so that the slope obtained in the linear regression is the convergence rate.
This process is done for the different values of $\beta\in(0,1/2)$ that we consider, and the resulting rate is shown in Figure \ref{fig:mckean_rates1}.

\begin{figure}[htbp]  
\figtitle{Theoretical and empirical rate of convergence}
    \begin{minipage}{0.49\textwidth}
    \centering
    \includegraphics[width=0.99\textwidth]{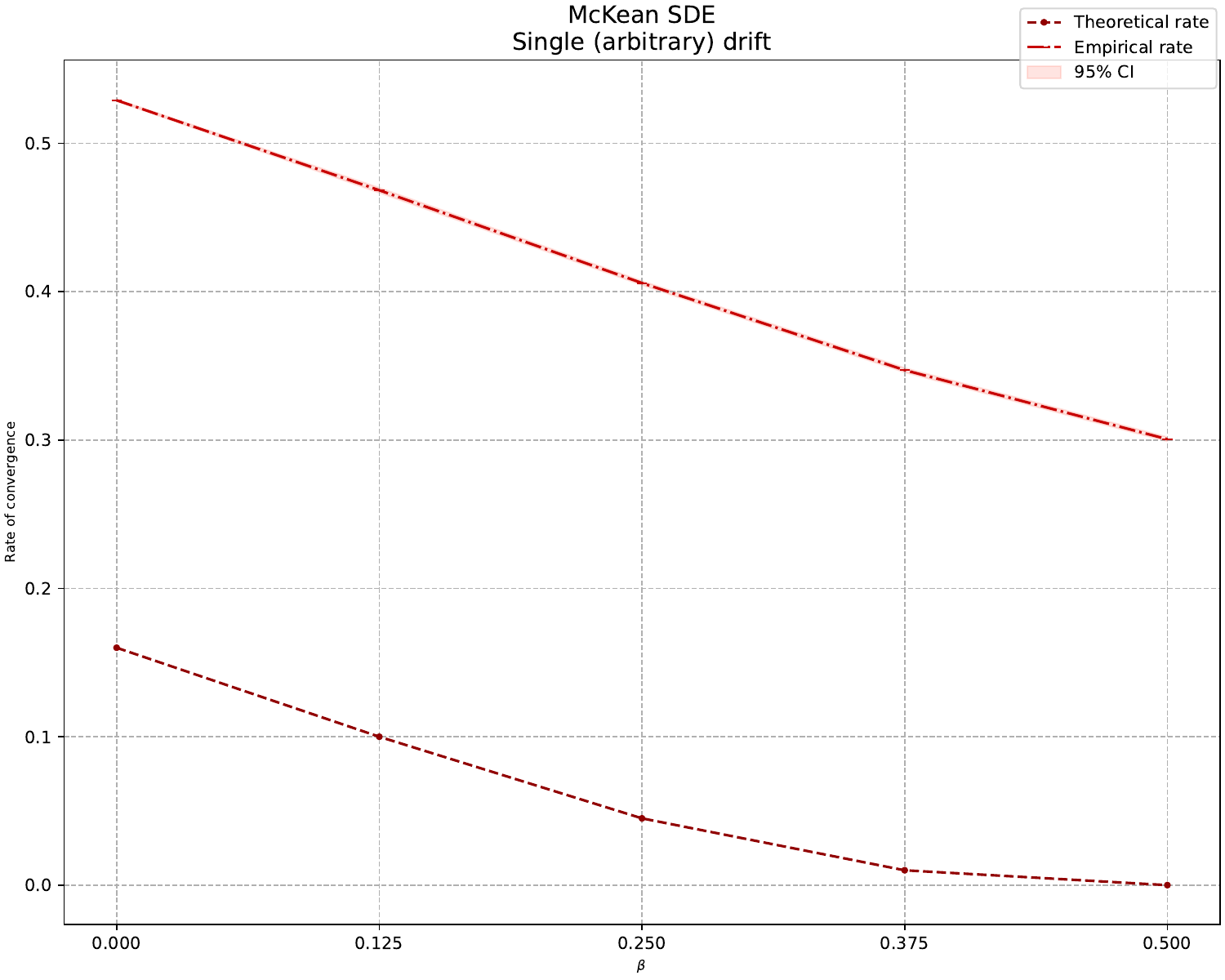}
\end{minipage}
\begin{minipage}{0.49\textwidth}
    \centering
    \includegraphics[width=0.99\textwidth]{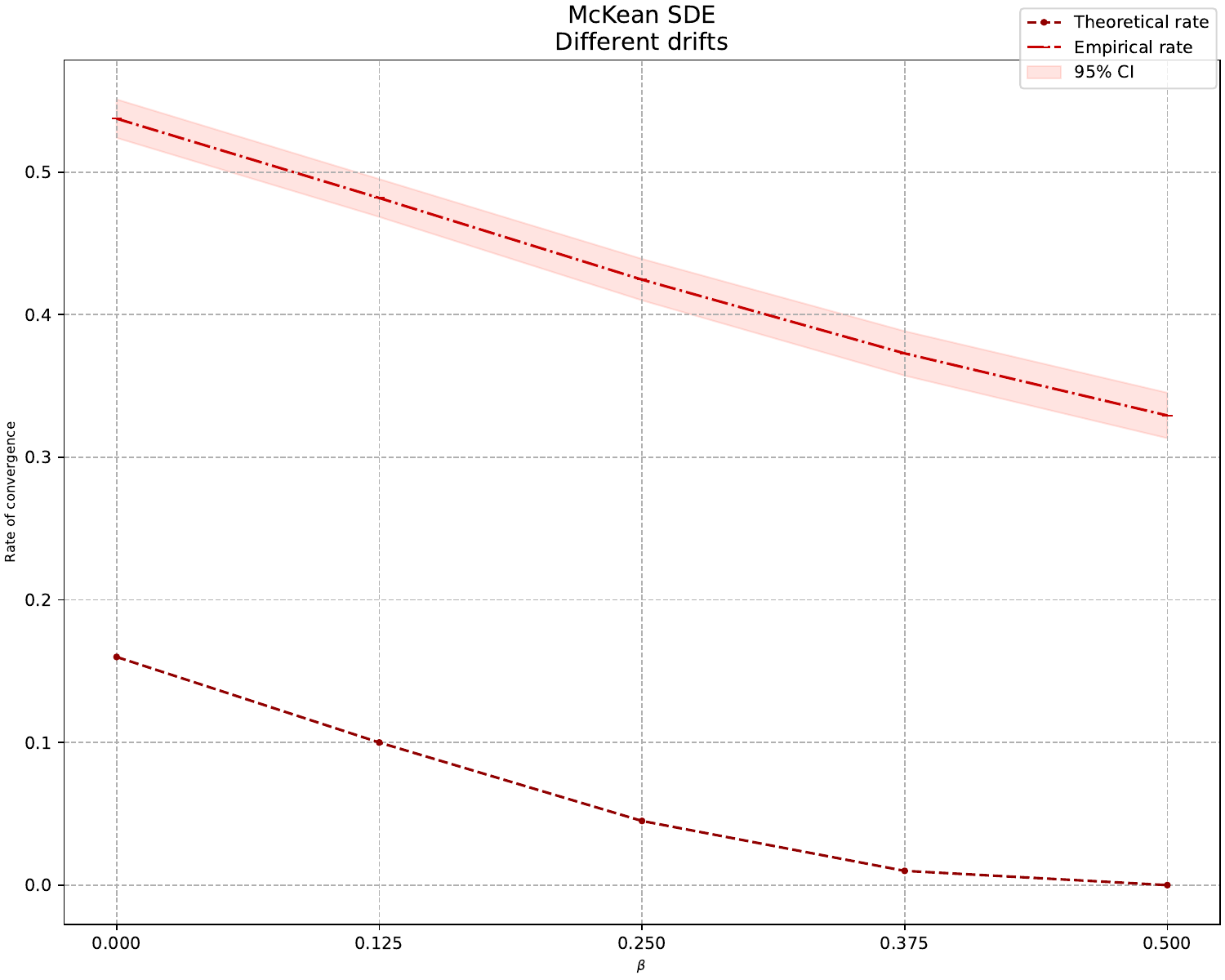}
\end{minipage} 
 \caption{Plot of theoretical (broken line) and empirical (dotted-broken line) convergence rate (vertical axis) against parameter $\beta$ (horizontal axis) obtained with  $F(\hat\rho^N)b^N$ with $b\in \mathcal C^{-\beta}$, $X_0 \sim \mathcal N(0, 1)$ and $F(x) = \sin(x)$. Left panel: we used a singular fixed drift, and  10,000 sample paths of a Brownian motion, run  40 times to obtain the empirical rate and its 95\% confidence inteval (shaded area around empirical rate).  Right panel:  same as before but we used multiple drifts, in particular 40 drifts (a different drift  for evey run).}
  \label{fig:mckean_rates1}
\end{figure}

We remark that using the fBm to produce the approximation $b^N$ makes the drift inherently random and effectively each time we generate a new fBm we are using a different drift which we claim lives in the same space $C^{-\beta}$ with $-\beta < H - 1$ where  $H$ is the Hurst parameter.
However, it is important to notice that this technique is not without flaws,
since the approximation procedure of the distributional drift uses the scale parameter $1/N$ and this can have a `regularization' effect which affects how volatile the approximation $b^N$ of the distribution is. 
We can observe this by looking at Figure \ref{fig:mckean_rates1}.
In Figure \ref{fig:mckean_rates1}, left panel, we see an example where we ran the Euler-Maruyama scheme 40 times (i.e\ 40 runs), with $10^4$ sample paths each run and a fixed drift for each of the values of $\beta =0.01,  0.125, 0.25, 0.375, 0.49$.
Fixing a drift and running the Euler scheme means that we use a single trajectory of fBm for each $\beta$, transform it into our drift $b$, and use different Brownian motion drivers in each of the runs.
In Figure \ref{fig:mckean_rates1}, left panel,  we can observe that the 95\% confidence interval is barely visible, meaning that the rates obtained for the same drift are not far apart considering different Brownian motions.
On the other hand, in Figure \ref{fig:mckean_rates1}, right panel  we do not fix the drift,  but instead we choose 40 different dritfs $b$ for each of the 40 runs, we notice that within the same values of $\beta$ we have some variation of the numerical rate. Of course this is to be expected due to the random nature of our technique for drift generation, since what is really being used to compute the error are effectively different SDEs by virtue of having different drifts.
Nonetheless, the confidence interval is still small, which indicates that our way to generate drifts is meaningful.

The usage of fractional Brownian motion as a base for our drift gives us a powerful method to study drifts with different regularities, and the randomness involved in the generation of a path of a stochastic process  provides as many different examples of drifts as we want, all with the same regularity. This allows us to explore in depth the behaviour of our algorithm.
Another natural choice is to take $b = h'$, where $h$ is a deterministic Hölder continuous function such as the Weierstrass function. However, our smoothing procedure consisting of the convolution with the heat kernel seemed to produce approximate drifts $b^N$ that were too smooth to identify numerically the effect of the roughness of $b$. Further experiments in this direction are left for a future work.

We turn our attention to the comparison between the solution to the Fokker-Planck equation at terminal time, $\hat\rho^N_T$, and  the empirical density of $\hat X^{N, m}_T$. 
\begin{figure}[htbp]
\figtitle{MVSDE densities for $b \in C_TC^{-\beta}$ for different non-linearity $F$}
    \begin{minipage}{0.49\textwidth}  
    \centering    
    \figsubtitle{$\beta = 0.49$}
    \includegraphics[width=0.99\textwidth]{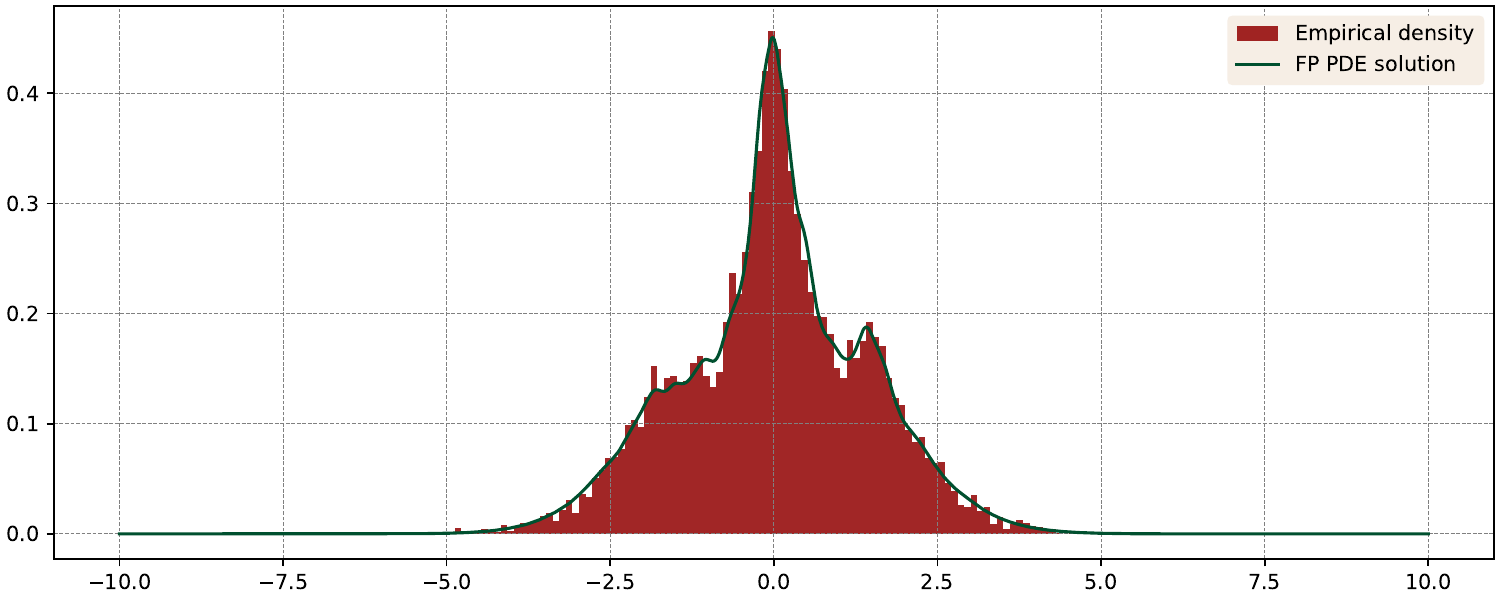}\par\smallskip
    \includegraphics[width=0.99\textwidth]{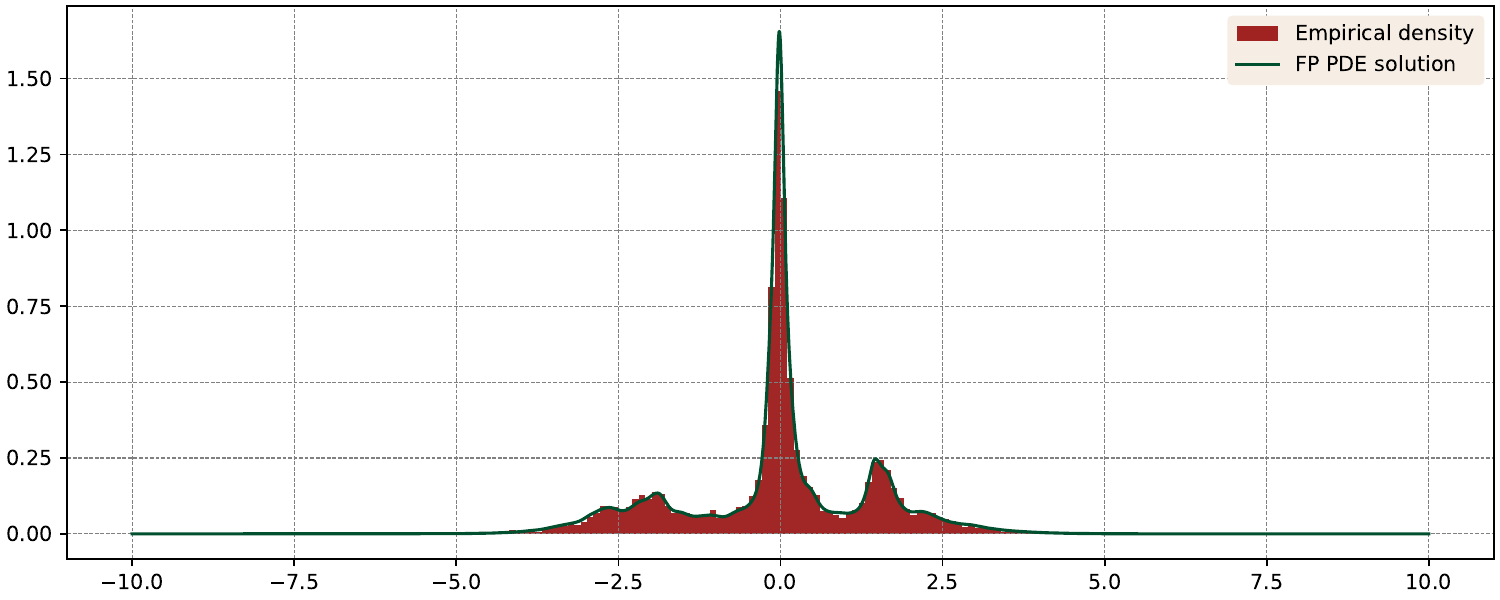}\par\smallskip
    \includegraphics[width=0.99\textwidth]{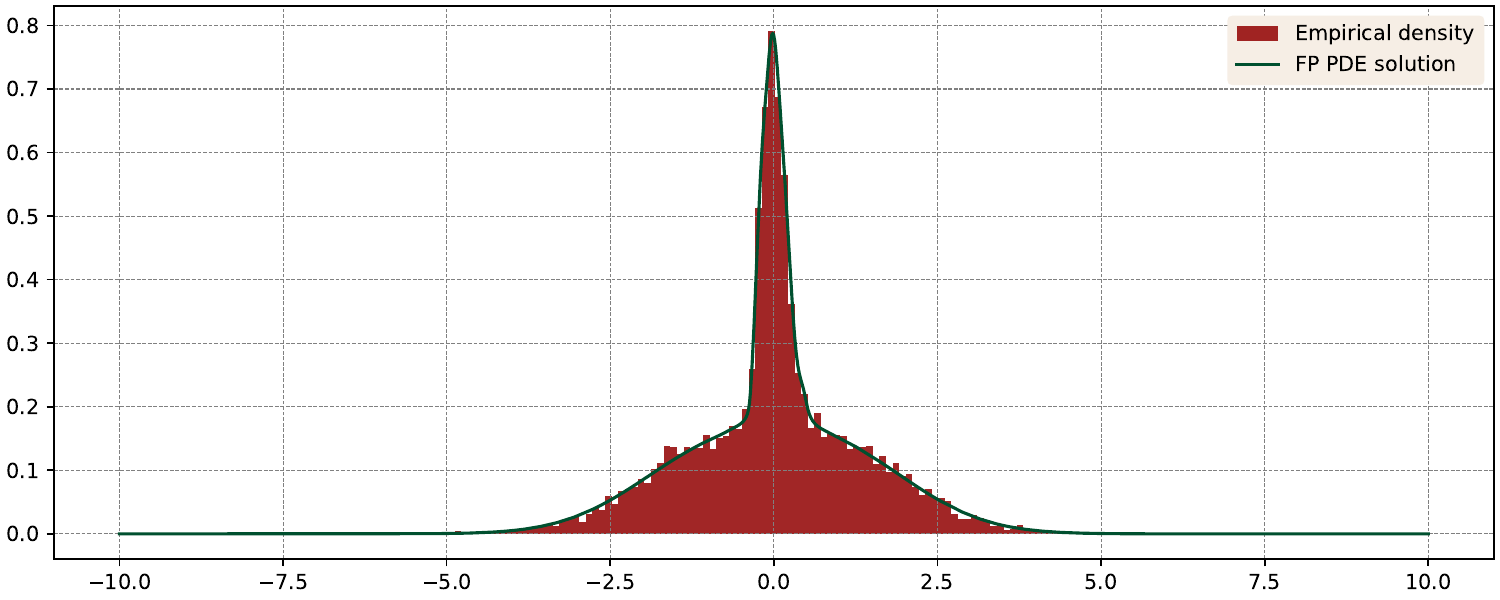}\par\smallskip
    \includegraphics[width=0.99\textwidth]{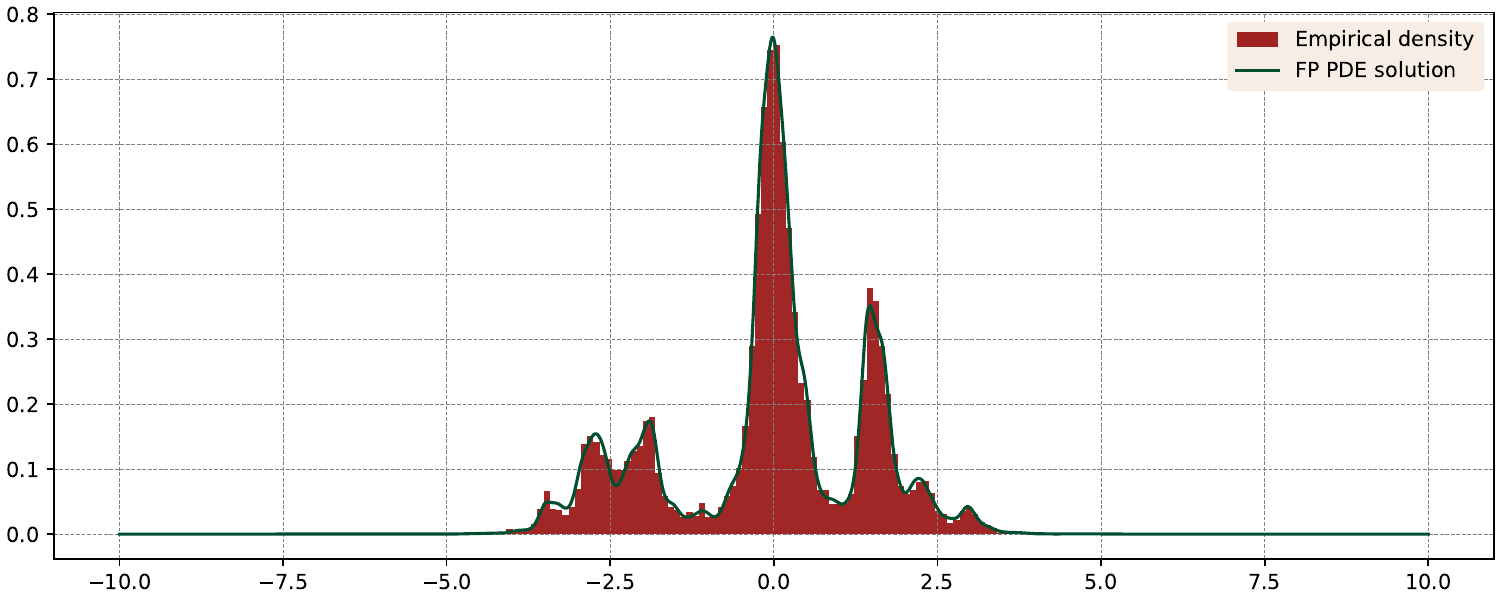}
    \end{minipage} 
    \begin{minipage}{0.49\textwidth}
     \centering
    \figsubtitle{$\beta = 0.01$}
    \includegraphics[width=0.99\textwidth]{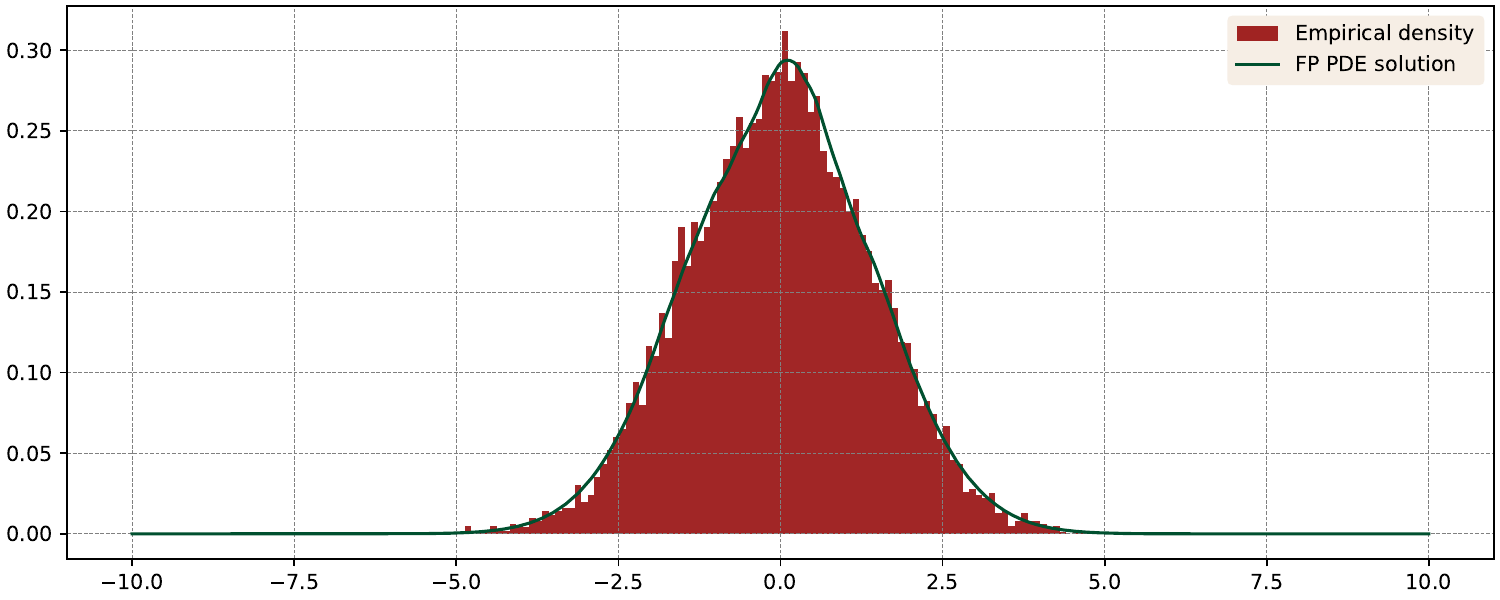}\par\smallskip
    \includegraphics[width=0.99\textwidth]{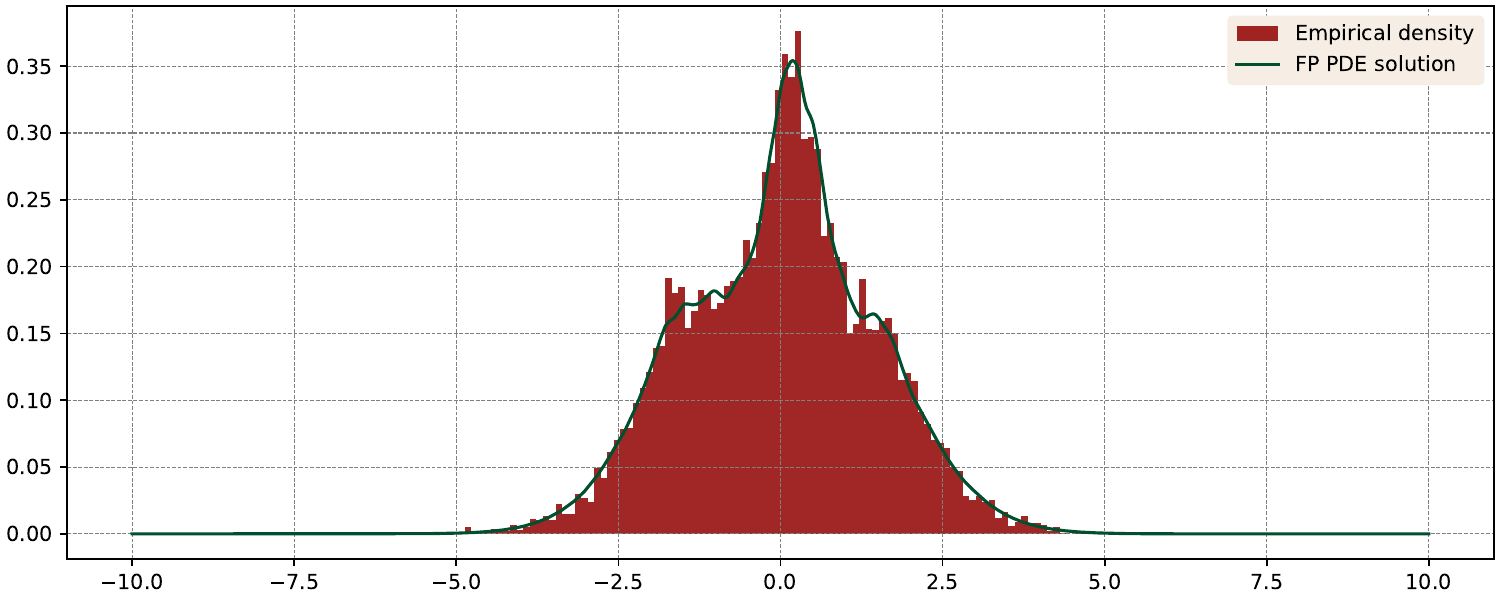}\par\smallskip
    \includegraphics[width=0.99\textwidth]{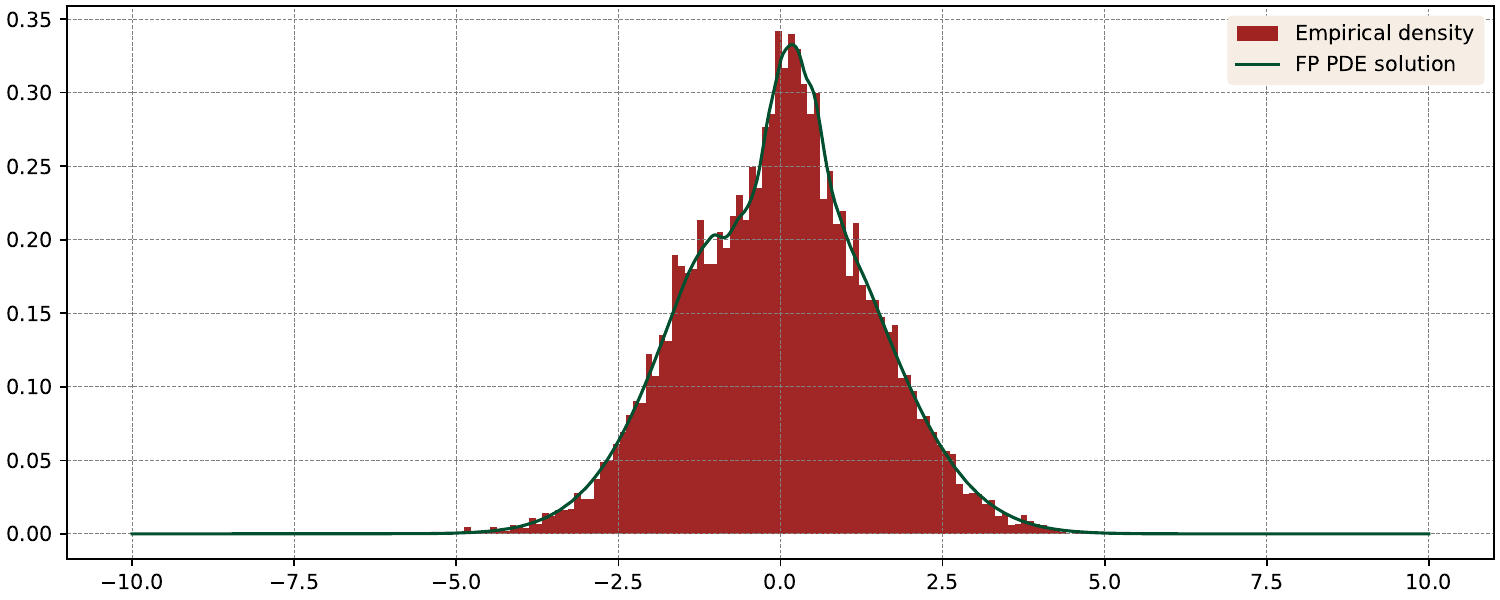}\par\smallskip
    \includegraphics[width=0.99\textwidth]{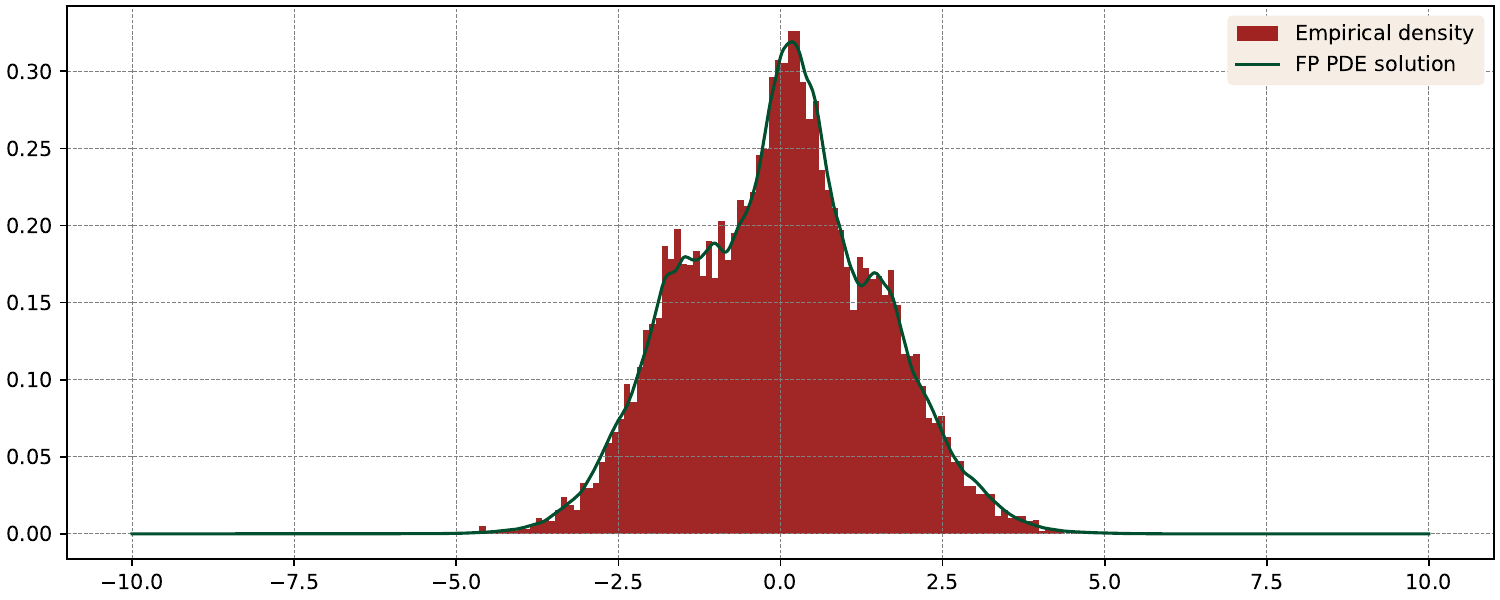}
    \end{minipage}    
    \caption{Representation of $\hat\rho^N$ for $b^N \in C^{-\beta}$ with different values of $F$ and of $\beta$. Left panels $\beta= 0.49$, right panels $\beta= 0.01$. From top to bottom $F(x) = \sin(x)$, $F(x) = 5\sin(x)$, $F(x) = 1/(1 + \exp(-100(x - 0.2)))$, $F(x) = \cos(x)$. SDE solved with $2^{11}$ time steps and $10^4$ sample paths. PDE solved on a grid with $2^{11}$ time steps and $4\times10^3$ points in $x \in [-10, 10]$.}
    \label{fig:densities_mckean1}
\end{figure}
In  Figure \ref{fig:densities_mckean1} we observe the comparison between $\hat\rho^N_T$ (solid line) and the density histogram of $\hat X^{N,m}_T$ (histogram bars)  for 4 different examples of nonlinearity $F$ and two different values of $\beta$.  This allows us to  explore the effect of changing $F$ on the law density of the solution on the same instance of the approximation $b^N$. 
In each example, the solution to the McKean-Vlasov equation is obtained through an Euler-Maruyama scheme with $2^{11}$ time steps, and $10^4$ sample paths, whereas the Fokker-Planck equation is obtained with $2^{11}$ time steps. The number of points in the $x$ variable is  $4\times10^{3}$.
We can visually attest a close fit of the solution of the Fokker-Planck PDE with the histogram. In order to check this fit,  we proceed with a Kolmogorov-Smirnov test for the goodness of fit, whose results are shown in Table \ref{table:ks2}.
\begin{table}[ht]
    \centering
    \begin{tabular}{c|cc|cc}
	& $\beta= 0.49 $ & &  $\beta = 0.01$ \\
        \hline					
         Time steps $m$  &K-S statistic & $p$-value  &K-S statistic & $p$-value \\
        \hline
        $2^7$  &0.0065
	&0.7861
	& 0.0055
	& 0.9186
\\
        $2^9$  & 0.0059
	& 0.8772
	& 0.0057
	& 0.9001
\\
        $2^{11}$ & 0.0061
	& 0.8514
	& 0.0057
	& 0.8964
\\
        \hline
\end{tabular}\\[5pt]
    \caption{Results of the Kolmogorov-Smirnov test for the goodness of fit of the PDE approximation of the law density with respect to the empirical density for McKean-Vlasov SDEs with drift in the space $C^{-\beta}$ for $\beta=0.49, 0.01$ and time steps $ 2^7, 2^9, 2^{11}$ and $F(x) = \cos(x)$.}
    \label{table:ks2}
\end{table}
As we can see, the Kolmogorov-Smirnov test does not allow rejecting the hypothesis that the sample of the solution has a density given by the solution of the Fokker-Planck PDE. Notice that the $p$-values for $\beta=0.01 $ are higher than those for $\beta=0.49$, and this is probably due to the fact that the original $b$ drift is smoother in this case.  
Notice moreover that the $p$-value seems to decrease slightly as we increase the number of time steps, which may at first seem to contradict the fact that both approximation of the SDE and the PDE must be better as we increase the time steps.
However, one must remember that, as the amount of time steps $m$ on the iterative methods increases,  the equation will be more difficult to solve since the approximation $b^{N(m)}$ will be rougher. Thus, as $m$ increases,  there may also be a need to increase the resolution in space of the grid, which instead is kept fixed in our implementations.

\vspace{10pt}
{\noindent\bf Acknowledgement:} The author E.\ Issoglio acknowledges partial financial support from PRIN-PNRR2022 (P20224TM7Z) CUP: D53D23018780001 and from EU - Next Generation EU - PRIN2022 (202277N5H9) CUP: D53D23005670006.


\printbibliography

\end{document}